\documentclass[12pt,leqno]{amsart}
\usepackage[latin9]{inputenc}
\usepackage{verbatim}
\usepackage{amsthm}
\usepackage{amstext}
\usepackage{amssymb}
\usepackage[dvipdfmx]{graphicx}
\usepackage{xcolor}
\usepackage{soul}

\usepackage{mathrsfs}
\usepackage{amscd}

\makeatletter
\numberwithin{equation}{section}

\usepackage{amsthm}\usepackage{amscd}\usepackage{bm}

\theoremstyle{plain}
\newtheorem{main theorem}{Main Theorem}
\newtheorem{theorem}{Theorem}[section]
\newtheorem{lemma}[theorem]{Lemma}

\newtheorem{corollary}[theorem]{Corollary}
\newtheorem{proposition}[theorem]{Proposition}
\newtheorem{claim}[theorem]{Claim}

\theoremstyle{definition}
\newtheorem{definition}[theorem]{Definition}
\newtheorem{remark}[theorem]{Remark}
\newtheorem{example}[theorem]{Example}

\newtheorem{condition}[theorem]{Condition}
\newtheorem{data}[theorem]{Data}

\newcommand{\mdim}{\mathrm{mdim}}
\newcommand{\diam}{\mathrm{diam}}
\newcommand{\widim}{\mathrm{Widim}}
\newcommand{\dist}{\mathrm{dist}}
\newcommand{\supp}{\mathrm{supp}}
\newcommand{\norm}[1]{\left|\!\left|#1\right|\!\right|}

\newcommand{\rdim}{\mathrm{rdim}}
\newcommand{\var}{\mathrm{var}}
\newcommand{\ver}{\mathrm{Ver}}

\makeatother

\begin{document}

\title[Double variational principle for mean dimension with potential]{Double variational principle for mean dimension with potential}

\author{Masaki Tsukamoto}



\date{\today}

\thanks{I was partially supported by JSPS KAKENHI 18K03275.}

\maketitle

\begin{abstract}
This paper contributes to the mean dimension theory of dynamical systems.
We introduce a new concept called mean dimension with potential and develop a variational principle for it.
This is a mean dimension analogue of the theory of topological pressure.
We consider a minimax problem for the sum of rate distortion dimension and the integral of a potential function.
We prove that the minimax value is equal to the mean dimension with potential for a dynamical system having the marker property.
The basic idea of the proof is a dynamicalization of geometric measure theory.
\end{abstract}

\section{Introduction}  \label{section: introduction}

\subsection{Backgrounds}  \label{subsection: backgrounds}

This paper is a continuation of the project \cite{Lindenstrauss--Tsukamoto rate distortion, Lindenstrauss--Tsukamoto double VP}, which 
aims to inject ergodic-theoretic ideas into mean dimension theory by constructing new variational principles.
The purpose of the present paper is to introduce a new quantity called \textit{mean dimension with potential} and 
develop a variational principle for it.
This is a mean dimension analogue of the theory of topological pressure.

A pair $(\mathcal{X},T)$ is called a \textbf{dynamical system}
if $\mathcal{X}$ is a compact metrizable space and $T:\mathcal{X}\to \mathcal{X}$
is a homeomorphism.
Gromov \cite{Gromov} defined a topological invariant of dynamical systems called mean dimension 
(denoted by $\mdim(\mathcal{X},T)$), which estimates how many parameters per iterate we need to describe the orbits of 
the system $(\mathcal{X}, T)$.
Several applications and interesting relations with other subjects have been found over the last two decades
\cite{Lindenstrauss--Weiss, Lindenstrauss, Gutman Jaworski theorem, 
Matsuo--Tsukamoto, Gutman--Lindenstrauss--Tsukamoto, Li--Liang, Tsukamoto Brody, Meyerovitch--Tsukamoto, 
Gutman--Tsukamoto minimal, Gutman--Qiao--Tsukamoto}.
However, before our paper \cite{Lindenstrauss--Tsukamoto rate distortion} appeared, 
the theory of mean dimension lacked an important ingredient -- ergodic theory (in particular, invariant measures).
The paper \cite{Lindenstrauss--Tsukamoto rate distortion} discovered
a close relation between mean dimension and \textit{rate distortion theory}, which is a foundation 
of lossy data compression method.
This was further developed by \cite{Lindenstrauss--Tsukamoto double VP}.
They enable us to inject ergodic-theoretic concepts into mean dimension.

The following two theories are the main backgrounds of the present paper:

\begin{itemize}
    \item \textbf{Variational principle for topological pressure \cite{Walter}:}
    Let $(\mathcal{X}, T)$ be a dynamical system with a continuous function (called \textit{potential}) $\varphi:\mathcal{X}\to \mathbb{R}$.
    Then we can define the \textbf{topological pressure} $P(T,\varphi)$, which is a generalization of the 
    \textbf{topological entropy} $h_{\mathrm{top}}(T)$ in the sense that $h_{\mathrm{top}}(T) = P(T, 0)$.
    Let $\mathscr{M}^T(\mathcal{X})$ be the set of $T$-invariant Borel probability measures on $\mathcal{X}$.
    The variational principle states that \cite[\S 9.3]{Walter book}
    \[  P(T,\varphi) = \sup_{\mu\in \mathscr{M}^T(\mathcal{X})} \left(h_{\mu}(T) + \int_{\mathcal{X}} \varphi \, d\mu\right). \]
    Here $h_{\mu}(T)$ is the ergodic-theoretic entropy.
    When $\varphi =0$, this specializes to the variational principle for topological entropy 
    \cite{Goodwyn, Dinaburg, Goodman}:
    \begin{equation}  \label{eq: VP for entropy}
        h_{\mathrm{top}}(T) = \sup_{\mu\in \mathscr{M}^T(\mathcal{X})} h_\mu(T).
    \end{equation}    

    \item \textbf{Double variational principle for mean dimension \cite{Lindenstrauss--Tsukamoto double VP}:}
    Let $(\mathcal{X},T)$ be a dynamical system. We denote by $\mathscr{D}(\mathcal{X})$ the set of metrics (i.e. distance functions)
    on $\mathcal{X}$ compatible with the topology. Take a metric $\mathbf{d} \in \mathscr{D}(\mathcal{X})$ and an invariant probability measure
    $\mu\in \mathscr{M}^T(\mathcal{X})$.
    Let $X$ be the random variable taking values in $\mathcal{X}$ and obeying the distribution $\mu$.
    Consider the stochastic process $\{T^n X\}_{n\in \mathbb{Z}}$ and let $R(\mathbf{d}, \mu, \varepsilon)$, $\varepsilon >0$, be the 
    \textit{rate distortion function} of this process.
    This evaluates how many bits per iterate 
    we need to describe the process within the distortion (with respect to $\mathbf{d}$) bounded by $\varepsilon$.
    We will review the definition of $R(\mathbf{d},\mu,\varepsilon)$ in 
    \S \ref{subsection: rate distortion theory}.
    Following Kawabata--Dembo \cite{Kawabata--Dembo}, we introduce the \textbf{upper and lower rate distortion dimensions} 
    by\footnote{Throughout the paper we assume that the base of the logarithm is two.} 
    \begin{equation} \label{eq: rate distortion dimension}
      \overline{\rdim}(\mathcal{X}, T, \mathbf{d}, \mu) = \limsup_{\varepsilon\to 0} \frac{R(\mathbf{d},\mu,\varepsilon)}{\log(1/\varepsilon)}, 
      \quad 
      \underline{\rdim}(\mathcal{X},T,\mathbf{d},\mu) = \liminf_{\varepsilon \to 0} \frac{R(\mathbf{d},\mu,\varepsilon)}{\log(1/\varepsilon)}.         
    \end{equation}
    When the upper and lower limits coincide, we denote the common value by $\rdim(\mathcal{X},T,\mathbf{d},\mu)$.
    
    A dynamical system $(\mathcal{X},T)$ is said to have the \textbf{marker property} if for any $N>0$ there exists an open set 
    $U\subset \mathcal{X}$ satisfying 
    \[  U\cap T^{n} U = \emptyset \quad (1\leq n\leq N), \quad 
         \mathcal{X} = \bigcup_{n\in \mathbb{Z}} T^{n} U. \]
    For example, free minimal systems and their extensions satisfy this condition.
    
    The \textit{double variational principle} \cite[Theorem 1.1]{Lindenstrauss--Tsukamoto double VP} 
    states that if a dynamical system $(\mathcal{X}, T)$ has the marker property then 
    its mean dimension is given by 
    \begin{equation}  \label{eq: double VP}
      \begin{split}
      \mdim(\mathcal{X},T) & = \min_{\mathbf{d}\in \mathscr{D}(\mathcal{X})} \sup_{\mu\in \mathscr{M}^T(\mathcal{X})} 
      \overline{\rdim}(\mathcal{X},T,\mathbf{d}, \mu)  \\
      & = \min_{\mathbf{d}\in \mathscr{D}(\mathcal{X})} \sup_{\mu\in \mathscr{M}^T(\mathcal{X})} 
      \underline{\rdim}(\mathcal{X},T,\mathbf{d}, \mu).
      \end{split}
    \end{equation}
    Here ``$\min$'' indicates that the minimum is attained by some metric $\textbf{d}$.
    The main difference between (\ref{eq: double VP}) and the standard variational principle (\ref{eq: VP for entropy})
    is that $h_{\mathrm{top}}(T) = \sup_{\mu} h_\mu(T)$ is a maximization problem with respect to the single variable $\mu$ wheres
    (\ref{eq: double VP}) is a minimax problem with respect to the two variables $\mathbf{d}$ and $\mu$.
    By the word ``double'' we emphasize the point that there exist two variables $\mathbf{d}$ and $\mu$ playing different roles.
\end{itemize}

We will develop a fusion of the above two theories.

\subsection{Mean dimension with potential}  \label{subsection: mean dimension with potential}

We introduce a mean dimension analogue of topological pressure in this subsection.
Throughout the paper we assume that simplicial complexes are finite (namely, they have only finitely many simplexes). 
Let $P$ be a simplicial complex. For $a\in P$ we define the \textbf{local dimension} $\dim_a P$ 
as the maximum of $\dim \Delta$ where $\Delta\subset P$ is a simplex of $P$ containing $a$.
See Figure \ref{figure: local dimension}.

\begin{figure}[h] 
    \centering
    \includegraphics[width=3.0in]{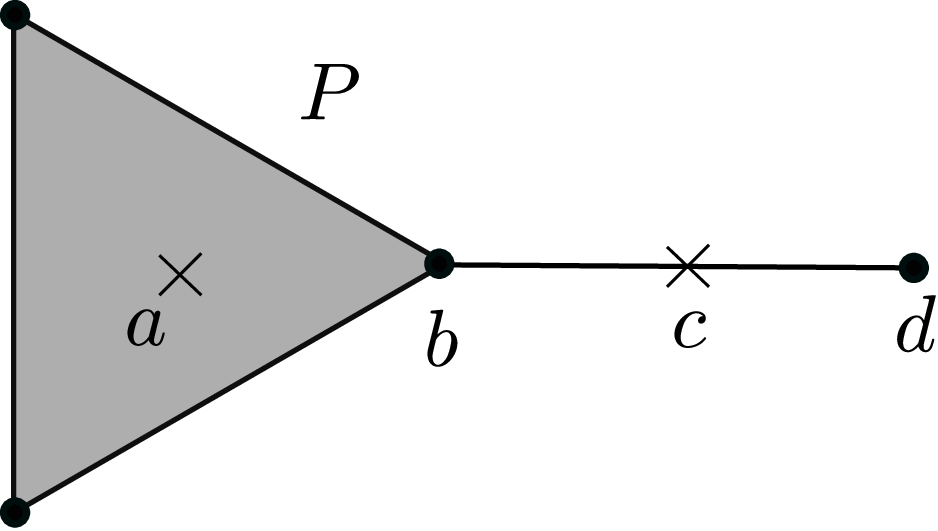}
    \caption{Here $P$ has four vertexes (denoted by dots), four $1$-dimensional simplexes and one $2$-dimensional simplex. 
    The points $b$ and $d$ are vertexes of $P$ wheres $a$ and $c$ are not. 
    We have $\dim_a P = \dim_b P =2$ and $\dim_c P = \dim_d P =1$.}    \label{figure: local dimension}
\end{figure}

Let $(\mathcal{X}, \mathbf{d})$ be a compact metric space and $f:\mathcal{X}\to \mathcal{Y}$ a continuous map into some topological space 
$\mathcal{Y}$.
For $\varepsilon>0$ we call the map $f$ an \textbf{$\varepsilon$-embedding} if $\diam f^{-1}(y) < \varepsilon$ for all $y\in \mathcal{Y}$.
Let $\varphi: \mathcal{X}\to \mathbb{R}$ be a continuous function.
We define the \textbf{$\varepsilon$-width dimension with potential} by
\begin{equation} \label{eq: widim with potential}
   \begin{split}
     & \widim_\varepsilon(\mathcal{X}, \mathbf{d}, \varphi)  \\ 
     & =  \inf\left\{\max_{x\in \mathcal{X}} \left(\dim_{f(x)} P + \varphi(x)\right) \middle|
    \parbox{3in}{\centering $P$ is a simplicial complex and $f:\mathcal{X}\to P$ is an $\varepsilon$-embedding}\right\}.  
   \end{split}
\end{equation}    

Let $T:\mathcal{X}\to \mathcal{X}$ be a homeomorphism.
For $N>0$ we define a metric $\mathbf{d}_N$ and a function $S_N\varphi$ on $\mathcal{X}$ by
\begin{equation}  \label{eq: dynamical metric and function}
    \mathbf{d}_N(x,y) = \max_{0\leq n <N} \mathbf{d}(T^n x, T^n y) \quad (x, y\in \mathcal{X}), \quad 
    S_N\varphi (x) = \sum_{n=0}^{N-1} \varphi(T^n x) \quad (x\in \mathcal{X}). 
\end{equation}    
We define the \textbf{mean dimension with potential} by 
\begin{equation}  \label{eq: mean dimension with potential}
   \mdim(\mathcal{X},T,\varphi) = 
    \lim_{\varepsilon \to 0} \left(\lim_{N\to \infty} \frac{\widim_\varepsilon\left(\mathcal{X}, \mathbf{d}_N, S_N\varphi\right)}{N}\right). 
\end{equation}    
The limits exist because the quantity $\widim_\varepsilon\left(\mathcal{X}, \mathbf{d}_N S_N\varphi\right)$ is subadditive in $N$ and 
monotone in $\varepsilon$.    
The value of $\mdim\left(\mathcal{X}, T, \varphi\right)$ is independent of the choice of $\mathbf{d}$.
Namely it becomes a topological invariant of 
$(\mathcal{X}, T, \varphi)$.
So we drop $\mathbf{d}$ from the notation.
When $\varphi=0$, the above (\ref{eq: mean dimension with potential}) specializes to the standard mean dimension:
$\mdim(\mathcal{X}, T, 0) = \mdim(\mathcal{X},T)$.

\subsection{Statement of the main result}  \label{subsection: statement of the main result}

Recall that, for a dynamical system $(\mathcal{X}, T)$, we denote by $\mathscr{D}(\mathcal{X})$ and $\mathscr{M}^T(\mathcal{X})$
the sets of metrics and invariant probability measures on it respectively.
The following is our main theorem.

\begin{theorem}[\textbf{Main result}]  \label{main theorem}
Let $(\mathcal{X},T)$ be a dynamical system with the marker property and let 
$\varphi:\mathcal{X}\to \mathbb{R}$ be a continuous function. Then
\begin{equation} \label{eq: double VP with potential}
    \begin{split}
      \mdim(\mathcal{X}, T, \varphi) & = \min_{\mathbf{d}\in \mathscr{D}(\mathcal{X})} \sup_{\mu\in \mathscr{M}^T(\mathcal{X})}
       \left( \overline{\rdim}(\mathcal{X},T,\mathbf{d},\mu) + \int_{\mathcal{X}} \varphi \, d\mu \right) \\
       & =   \min_{\mathbf{d}\in \mathscr{D}(\mathcal{X})} \sup_{\mu\in \mathscr{M}^T(\mathcal{X})}
       \left( \underline{\rdim}(\mathcal{X},T,\mathbf{d},\mu) + \int_{\mathcal{X}} \varphi \, d\mu \right).                   
       \end{split}
\end{equation}    
\end{theorem}

\begin{remark}  \label{remark: marker property}
We conjecture that the marker property assumption in Theorem \ref{main theorem} is unnecessary.
Namely we conjecture that (\ref{eq: double VP with potential}) holds for any dynamical system $(\mathcal{X}, T)$ and
any continuous function $\varphi:\mathcal{X}\to \mathbb{R}$.
The proof of Theorem \ref{main theorem} shows that the inequality 
\begin{equation} \label{eq: one direction of double VP}
    \mdim(\mathcal{X}, T,\varphi) \leq  \inf_{\mathbf{d}\in \mathscr{D}(\mathcal{X})} \sup_{\mu\in \mathscr{M}^T(\mathcal{X})}
       \left( \underline{\rdim}(\mathcal{X},T,\mathbf{d},\mu) + \int_{\mathcal{X}} \varphi \, d\mu \right) 
\end{equation}       
holds without the marker property assumption.
So the problem is how to prove the reverse inequality.
\end{remark}

\begin{example}  \label{example: shift on the Hilbert cube}
Let $[0,1]^\mathbb{Z}$ be the infinite product of the unit interval $[0,1]$ index by integers.
Let $\sigma: [0,1]^\mathbb{Z}\to [0,1]^\mathbb{Z}$ be the shift.
Define $\varphi:[0,1]^\mathbb{Z}\to \mathbb{R}$ by 
\[ \varphi\left((x_n)_{n\in \mathbb{Z}}\right) = x_0. \]
Then it is easy to check 
\[ \mdim\left([0,1]^\mathbb{Z}, \sigma, \varphi\right) = 2. \]
Define a metric $\mathbf{d}$ on $[0,1]^\mathbb{Z}$ by 
\[ \mathbf{d}\left((x_n)_{n\in \mathbb{Z}}, (y_n)_{n\in \mathbb{Z}}\right) = \sum_{n\in \mathbb{Z}} 2^{-|n|} |x_n-y_n|. \]
We can check that for any invariant probability measure $\mu \in \mathscr{M}^\sigma\left([0,1]^\mathbb{Z}\right)$
\[  \overline{\rdim}\left([0,1]^\mathbb{Z},\sigma, \mathbf{d},\mu\right) \leq 1, \quad \int_{[0,1]^\mathbb{Z}} \varphi \, d\mu \leq 1. \]
In particular 
\[    \underline{\rdim}\left([0,1]^\mathbb{Z},\sigma, \mathbf{d},\mu\right) + \int_{[0,1]^\mathbb{Z}} \varphi \, d\mu \leq
     \overline{\rdim}\left([0,1]^\mathbb{Z},\sigma, \mathbf{d},\mu\right) + \int_{[0,1]^\mathbb{Z}} \varphi \, d\mu \leq 2. \]
On the other hand, the inequality (\ref{eq: one direction of double VP}) holds for all dynamical systems.
So we get 
\begin{equation*}
    \begin{split}
     \mdim\left([0,1]^\mathbb{Z}, \sigma, \varphi\right) = 2 
    & =  \sup_{\mu\in \mathscr{M}^\sigma\left([0,1]^\mathbb{Z}\right)} \left(\underline{\rdim}\left([0,1]^\mathbb{Z},\sigma, \mathbf{d},\mu\right)
    + \int_{[0,1]^\mathbb{Z}} \varphi\, d\mu\right) \\
   &   = \sup_{\mu\in \mathscr{M}^\sigma\left([0,1]^\mathbb{Z}\right)} \left(\overline{\rdim}\left([0,1]^\mathbb{Z},\sigma, \mathbf{d},\mu\right)
    + \int_{[0,1]^\mathbb{Z}} \varphi\, d\mu\right).
    \end{split}
\end{equation*}    
Indeed, we can check this more directly.
Let $\nu$ be the Lebesgue measure on $[0,1]$, and let $\nu_k$ $(k\geq 1)$ be the probability measure on $[0,1]$ defined by 
\[  \nu_k(A) = k\cdot \nu\left(A\cap \left[1-\frac{1}{k},1\right]\right). \]
We define an invariant probability measure $\mu_k$ on $[0,1]^\mathbb{Z}$ by 
$\mu_k = \nu_k^{\otimes \mathbb{Z}}$.  
Then 
\[  \rdim\left([0,1]^\mathbb{Z}, \sigma, \mathbf{d},\mu_k\right) = 1, \quad 
     \int_{[0,1]^\mathbb{Z}} \varphi  \, d\mu_k = 1-\frac{1}{2k}.  \]
Hence 
\[  \rdim\left([0,1]^\mathbb{Z}, \sigma, \mathbf{d},\mu_k\right) +  \int_{[0,1]^\mathbb{Z}} \varphi  \, d\mu_k
    \to 2 \quad (k\to \infty).  \]
This example is very simple.
We plan to study a much deeper example in a future paper.
See \S \ref{subsection: future directions}.
\end{example}

\subsection{Main ingredients of the proof}  \label{subsection: main ingredients of the proof}

The proof of Theorem \ref{main theorem} follows the line of ideas developed in \cite{Lindenstrauss--Tsukamoto double VP}.
The basic idea is a \textit{dynamicalization} of geometric measure theory.
We consider the following four fundamental ingredients of geometric measure theory:

\begin{itemize}
  \item Minkowski dimension.
  \item Hausdorff dimension.
  \item Frostman's lemma \cite{Howroyd}: For a compact metric space $(\mathcal{X}, \mathbf{d})$ we can construct a probability measure on it 
           satisfying the ``scaling law'' of degree given by the Hausdorff dimension.
  \item Pontrjagin--Schnirelmann's theorem \cite{Pontrjagin--Schnirelmann}: 
     For a compact metrizable space $\mathcal{X}$ we can construct a metric $\mathbf{d}$ on it 
     for which the upper Minkowski dimension is equal to the topological dimension.
\end{itemize}

The paper \cite{Lindenstrauss--Tsukamoto double VP} developed dynamical analogues of these ingredients.
A dynamical version of Minkowski dimension is \textit{metric mean dimension} introduce by 
Lindenstrauss--Weiss \cite{Lindenstrauss--Weiss}.
A corresponding ``dynamical Pontrjagin--Schnirelamann's theorem'' was proved in \cite{Lindenstrauss--Tsukamoto double VP},
developing the idea of Lindenstrauss \cite{Lindenstrauss}.
The paper \cite{Lindenstrauss--Tsukamoto double VP} introduced the notion \textit{mean Hausdorff dimension}
(a dynamical version of Hausdorff dimension) and established ``dynamical Frostman's lemma''.
Combining these ingredients, we proved the double variational principle
(\ref{eq: double VP}) in \cite{Lindenstrauss--Tsukamoto double VP}.

The main point of the proof of Theorem \ref{main theorem}
is how to combine the information of potential function to the above objects.
It is somehow surprising (at least for the author) that the argument of \cite{Lindenstrauss--Tsukamoto double VP}
is so robust that we can naturally adapt everything to the setting ``with potential''.
Probably the most important contribution of the present paper is that we clarify how to 
define mean topological/Minkowski/Hausdorff dimensions \textit{with potential}.
The definition of mean (topological) dimension with potential was already given in \S \ref{subsection: mean dimension with potential}.
The other two are defined as follows. Let $(\mathcal{X},\mathbf{d})$ be a compact metric space with a continuous function 
$\varphi:\mathcal{X}\to \mathbb{R}$.

\begin{itemize}
    \item \textbf{Metric mean dimension with potential:}
    For $\varepsilon>0$ we set 
    \begin{equation} \label{eq: covering number with potential}
       \#\left(\mathcal{X},\mathbf{d},\varphi,\varepsilon\right) = 
  \inf\left\{ \sum_{i=1}^n (1/\varepsilon)^{\sup_{U_i} \varphi} \middle|\,
   \parbox{3in}{\centering  $\mathcal{X} = U_1\cup \dots \cup U_n$ is an open cover with $\diam\, U_i < \varepsilon$
                   for all $1\leq i\leq n$}\right\}.  
    \end{equation}               
   Given a homeomorphism $T:\mathcal{X}\to \mathcal{X}$, we define a metric $\mathbf{d}_N$ and a function $S_N\varphi$ 
   on $\mathcal{X}$ by 
   (\ref{eq: dynamical metric and function}) in \S \ref{subsection: mean dimension with potential}.
   We set 
   \[  P(\mathcal{X},T,\mathbf{d},\varphi,\varepsilon) = \lim_{N\to \infty} 
   \frac{\log  \#\left(\mathcal{X}, \mathbf{d}_N,S_N\varphi,\varepsilon\right)}{N}. \]
   This limit exists because $\log  \#\left(\mathcal{X},\mathbf{d}_N,S_N\varphi,\varepsilon\right)$ is subadditive in $N$.
   We define the \textbf{upper and lower metric mean dimensions with potential} by 
   \begin{equation*}
     \begin{split}
       \overline{\mdim}_{\mathrm{M}}(\mathcal{X},T,\mathbf{d},\varphi) &= 
       \limsup_{\varepsilon\to 0} \frac{P(\mathcal{X},T,\mathbf{d},\varphi,\varepsilon)}{\log(1/\varepsilon)}, \\
       \underline{\mdim}_{\mathrm{M}}(\mathcal{X},T,\mathbf{d},\varphi) & = 
       \liminf_{\varepsilon\to 0} \frac{P(\mathcal{X},T,\mathbf{d},\varphi,\varepsilon)}{\log(1/\varepsilon)}. 
     \end{split}
   \end{equation*}
     When the upper and lower limits coincide, we denote the common value by 
     $\mdim_{\mathrm{M}}(\mathcal{X},T,\mathbf{d},\varphi)$.
   \item \textbf{Mean Hausdorff dimension with potential:}
   For $\varepsilon>0$ and $s\geq \max_{\mathcal{X}} \varphi$ we set 
   \[ \mathcal{H}^s_\varepsilon(\mathcal{X},\mathbf{d},\varphi) = 
    \inf \left\{\sum_{i=1}^\infty \left(\diam E_i\right)^{s-\sup_{E_i} \varphi} \middle|
     \mathcal{X} = \bigcup_{i=1}^\infty E_i \text{ with } \diam E_i < \varepsilon 
                        \text{ for all $i\geq 1$}\right\}. \]
     Here we have used the convention that $0^0 = 1$ and $(\diam \, \emptyset)^s = 0$ for all $s\geq 0$.
     Note that this convention implies $\mathcal{H}^{\max_\mathcal{X} \varphi}_\varepsilon (\mathcal{X},\mathbf{d}, \varphi) \geq 1$. 
     We define $\dim_{\mathrm{H}}(\mathcal{X},\mathbf{d},\varphi,\varepsilon)$ as the supremum of $s\geq \max_{\mathcal{X}}\varphi$
     satisfying $\mathcal{H}^s_\varepsilon(\mathcal{X},\mathbf{d},\varphi) \geq 1$.
     Given a homeomorphism $T:\mathcal{X}\to \mathcal{X}$, we define the \textbf{mean Hausdorff dimension with potential}
     by 
     \[  \mdim_{\mathrm{H}}(\mathcal{X},T,\mathbf{d},\varphi) = 
         \lim_{\varepsilon\to 0} \left(\limsup_{N\to \infty} \frac{\dim_{\mathrm{H}}(\mathcal{X},\mathbf{d}_N,S_N\varphi, \varepsilon)}{N}\right). \]
     We can also define the \textbf{lower mean Hausdorff dimension with potential} 
     $\underline{\mdim}_{\mathrm{H}}(\mathcal{X},T,\mathbf{d},\varphi)$
     by replacing $\limsup_N$ with $\liminf_N$ in this definition. But we do not need this concept in the paper.    
\end{itemize}

It is well-known in the classical dimension theory that 
\[ \text{Topological dimension} \leq \text{Hausdorff dimension} \leq \text{Minkowski dimension}. \]
The following is its dynamical version (with potential).

\begin{theorem}[$=$ Theorem \ref{theorem: mean Hausdorff dimension bounds mean dimension restated}] 
\label{theorem: mean Hausdorff dimension bounds mean dimension}
\[  \mdim(\mathcal{X},T,\varphi) \leq \mdim_{\mathrm{H}}(\mathcal{X},T,\mathbf{d}, \varphi) \leq 
     \underline{\mdim}_{\mathrm{M}}(\mathcal{X},T,\mathbf{d},\varphi). \]
\end{theorem}

The following is a version of ``dynamical Frostman's lemma''.
It states that we can construct invariant probability measures capturing dynamical complexity of $(\mathcal{X},T,\mathbf{d},\varphi)$.

\begin{theorem}[$\subset$ Theorem \ref{theorem: dynamical Frostman's lemma refined}]  \label{theorem: dynamical Frostman's lemma}
Under a mild condition on $\mathbf{d}$ (called tame growth of covering numbers; see Definition \ref{def: tame growth of covering numbers})
\[  \mdim_{\mathrm{H}}(\mathcal{X},T,\mathbf{d},\varphi) \leq \sup_{\mu\in \mathscr{M}^T(\mathcal{X})} 
     \left(\underline{\rdim}(\mathcal{X},T,\mathbf{d},\mu) + \int_{\mathcal{X}} \varphi \, d\mu\right). \]
\end{theorem}

On the other hand, it is easy to show:

\begin{proposition}[$=$ Proposition \ref{prop: metric mean dimension bounds rate distortion dimension restated}] 
\label{prop: metric mean dimension bounds rate distortion dimension}
For any $\mu\in \mathscr{M}^T(\mathcal{X})$
\begin{equation*}
  \begin{split}
    \overline{\rdim}(\mathcal{X},T,\mathbf{d},\mu) + \int_{\mathcal{X}} \varphi \, d\mu  & 
    \leq \overline{\mdim}_{\mathrm{M}}(\mathcal{X}, T, \mathbf{d},  \varphi), \\
    \underline{\rdim}(\mathcal{X},T,\mathbf{d}, \mu) + \int_{\mathcal{X}} \varphi \, d\mu  &
     \leq \underline{\mdim}_{\mathrm{M}}(\mathcal{X},T, \mathbf{d}, \varphi).    
  \end{split}
\end{equation*}  
\end{proposition}

From the above three results (with a minor consideration on the tame growth of covering numbers condition\footnote{The inequality
(\ref{eq: mean dimension and rate distortion dimension})
holds for all metrics $\mathbf{d}$; see \S \ref{subsection: proof of Corollary of dynamical Frostman's lemma}.})

\begin{corollary}[$=$ Corollary \ref{corollary: mean dimension, rate distortion dimension and metric mean dimension restated}]
\label{corollary: mean dimension, rate distortion dimension and metric mean dimension}
\begin{equation} \label{eq: mean dimension and rate distortion dimension}
  \begin{split}
    \mdim(\mathcal{X}, T, \varphi) & \leq  \sup_{\mu\in \mathscr{M}^T(\mathcal{X})}
    \left(\underline{\rdim}(\mathcal{X},T, \mathbf{d}, \mu) + \int_{\mathcal{X}} \varphi \, d\mu\right)  \\
   &  \leq  \sup_{\mu\in \mathscr{M}^T(\mathcal{X})} \left(\overline{\rdim}(\mathcal{X},T,\mathbf{d}, \mu) +
    \int_{\mathcal{X}} \varphi \, d\mu\right)
   \leq \overline{\mdim}_{\mathrm{M}}(\mathcal{X}, T, \mathbf{d}, \varphi).   
  \end{split} 
\end{equation}
\end{corollary}

Now the following version of 
``dynamical Pontrjagin--Schnirelamann's theorem'' establishes Theorem \ref{main theorem}.

\begin{theorem}[$\subset$ Theorem \ref{theorem: dynamical PS theorem refined}] \label{theorem: dynamical PS theorem}
If $(\mathcal{X}, T)$ has the marker property then there exists a metric $\mathbf{d}\in \mathscr{D}(\mathcal{X})$
satisfying 
\[  \mdim(\mathcal{X}, T, \varphi) = \overline{\mdim}_{\mathrm{M}}(\mathcal{X},T,\mathbf{d},\varphi). \]
\end{theorem}

\begin{proof}[Proof of Theorem \ref{main theorem}]
The inequality (\ref{eq: mean dimension and rate distortion dimension}) holds for all metrics $\mathbf{d}$.
On the other hand, from Theorem \ref{theorem: dynamical PS theorem}, we can choose a metric $\mathbf{d}$
satisfying $\mdim(\mathcal{X},T,\varphi) = \overline{\mdim}_{\mathrm{M}}(\mathcal{X},T,\mathbf{d},\varphi)$.
Then, for this metric, we have 
 \begin{equation*}
  \begin{split}
    \mdim(\mathcal{X}, T, \varphi) & =  \sup_{\mu\in \mathscr{M}^T(\mathcal{X})}
    \left(\underline{\rdim}(\mathcal{X},T, \mathbf{d}, \mu) + \int_{\mathcal{X}} \varphi \, d\mu\right)  \\
   &  =  \sup_{\mu\in \mathscr{M}^T(\mathcal{X})} \left(\overline{\rdim}(\mathcal{X},T,\mathbf{d}, \mu) +
    \int_{\mathcal{X}} \varphi \, d\mu\right).
  \end{split} 
\end{equation*}
This proves Theorem \ref{main theorem}.
\end{proof}

We emphasize that only Theorem \ref{theorem: dynamical PS theorem} requires the marker property assumption.
Theorem \ref{theorem: mean Hausdorff dimension bounds mean dimension}, Theorem \ref{theorem: dynamical Frostman's lemma},
Proposition \ref{prop: metric mean dimension bounds rate distortion dimension} 
and Corollary \ref{corollary: mean dimension, rate distortion dimension and metric mean dimension}
hold for all dynamical systems.

\subsection{Future directions}  \label{subsection: future directions}

 This paper is devoted to the general theory of mean dimension with potential.
However our primary motivation is \textit{not} to develop the abstract theory.
Hopefully the theory of mean dimension with potential will shed a new light on the 
study of concrete examples as the topological pressure theory plays a crucial role 
in hyperbolic dynamics \cite{Bowen75}.
Here we briefly describe a possibility of such directions.

Let $\mathbb{C}P^N$ be the complex projective space with the Fubini--Study metric.
A holomorphic map $f:\mathbb{C}\to \mathbb{C}P^N$ is called a \textbf{Brody curve}
if it is one-Lipschitz.
This means that $f=[f_0:f_1:\dots:f_N]$ satisfies 
\[ |df|^2(z) := \frac{1}{4\pi} \left(\frac{\partial^2}{\partial x^2} + \frac{\partial^2}{\partial y^2}\right)
    \log \left(|f_0|^2+|f_1|^2+\dots+|f_N|^2\right) \leq 1. \]
Here $z=x+y\sqrt{-1}$ is the standard coordinate of $\mathbb{C}$.
Let $\mathcal{X}$ be the space of Bordy curves $f:\mathbb{C}\to \mathbb{C}P^N$ with the compact-open topology.
This is a compact metrizable space and the group $\mathbb{C}$ naturally acts on it:
\[ \mathbb{C}\times \mathcal{X}\to \mathcal{X}, \quad 
    (a, f(z))\mapsto f(z+a). \]
We denote the mean dimension of this action by $\mdim(\mathcal{X},\mathbb{C})$.

Gromov \cite[p.396 (c)]{Gromov} proposed the problem of estimating $\mdim(\mathcal{X},\mathbb{C})$.
The paper \cite{Tsukamoto Brody} found the following exact formula of $\mdim(\mathcal{X},\mathbb{C})$.
(The description below looks different from the formulation in \cite{Tsukamoto Brody}, but they
are equivalent.)
Let $\mathscr{M}^\mathbb{C}(\mathcal{X})$ be the set of Borel probability measures on $\mathcal{X}$
invariant under the $\mathbb{C}$-action.
Define a continuous function $\varphi:\mathcal{X}\to \mathbb{R}$ by 
\[  \varphi(f) = \frac{2(N+1)}{\pi} \int_{|z|<1} |df|^2 \, dxdy. \]
Then the mean dimension is given by 
\begin{equation}  \label{eq: mean dimension formula for Brody curves}
  \mdim(\mathcal{X},\mathbb{C}) = \sup_{\mu\in \mathscr{M}^{\mathbb{C}}(\mathcal{X})} \int_{\mathcal{X}} \varphi \, d\mu. 
\end{equation}
This formula looks mysterious.
Why is the mean dimension connected to the supremum of certain integral?
It seems that a deeper ergodic theoretic phenomena is hidden behind the formula.

We have been seeking a framework for understanding the formula (\ref{eq: mean dimension formula for Brody curves})
better.
Hopefully the theory of mean dimension with potential will provide 
such a framework\footnote{Although the present paper develops
the theory only for $\mathbb{Z}$-actions, we believe that 
everything can be generalized to $\mathbb{Z}^k$ or $\mathbb{R}^k$-actions without any significant difficulties.
We also would like to point out that 
the proof of the formula (\ref{eq: mean dimension formula for Brody curves}) in \cite{Tsukamoto Brody}
deeply uses metric structure (in particular, metric mean dimension). This is another indication that 
metric measure structure will be important in the future of the study of mean dimension.}.
(Notice that the right-hand side of the double variational principle (Theorem \ref{main theorem})
contains the same integral.)
We plan to study this direction in a future paper.

\subsection{Organization of the paper}  \label{subsection: organization of the paper}

In \S \ref{section: information theoretic preliminaries}
we prepare basic of mutual information and rate distortion function. 
We prove Theorem \ref{theorem: mean Hausdorff dimension bounds mean dimension}
and Proposition \ref{prop: metric mean dimension bounds rate distortion dimension}
in \S \ref{section: mean Hausdorff dimension bounds mean dimension}.
We prove Theorem \ref{theorem: dynamical Frostman's lemma}
(a version of dynamical Frostman's lemma) and Corollary \ref{corollary: mean dimension, rate distortion dimension and metric mean dimension}
in \S \ref{section: dynamical Frostman's lemma}.
We prove Theorem \ref{theorem: dynamical PS theorem} (a version of dynamical Pontrjagin--Schnirelamann's theorem)
in \S \ref{section: dynamical PS theorem}.
The arguments of \S \ref{section: dynamical PS theorem} are technically heavy.

\section{Information theoretic preliminaries}  \label{section: information theoretic preliminaries}

\subsection{Mutual information}  \label{subsection: mutual information}

Here we gather basic definitions and results on mutual information \cite[Chapter 2]{Cover--Thomas}.
We omit most of the proofs, which can be found in \cite[Section 2.2]{Lindenstrauss--Tsukamoto double VP}.
Throughout this subsection we fix a probability space $(\Omega, \mathbb{P})$ and assume that 
all random variables are defined on it.

Let $\mathcal{X}$ and $\mathcal{Y}$ be measurable spaces, and let $X$ and $Y$ be random variables taking values in 
$\mathcal{X}$ and $\mathcal{Y}$ respectively.
We want to define their \textbf{mutual information} $I(X;Y)$, which estimates the amount of information shared by $X$ and $Y$.

\textbf{Case 1:} Suppose $\mathcal{X}$ and $\mathcal{Y}$ are finite sets.
(We always assume that the $\sigma$-algebras of finite sets are the sets of all subsets.)
Then we define 
\begin{equation}  \label{eq: definition of mutual information}
   I(X;Y) = H(X) + H(Y) -H(X,Y) = H(X) - H(X|Y). 
\end{equation}   
More explicitly
\[
    I(X;Y) = \sum_{x\in \mathcal{X}, y\in \mathcal{Y}} \mathbb{P}(X=x, Y=y) \log \frac{\mathbb{P}(X=x, Y=y)}{\mathbb{P}(X=x) \mathbb{P}(Y=y)}.  
    \]   
Here we use the convention that $0\log(0/a)=0$ for all $a\geq 0$.

\textbf{Case 2:} In general, take measurable maps $f:\mathcal{X}\to A$ and $g:\mathcal{Y}\to B$ into 
finite sets $A$ and $B$. 
Then we can consider $I(f\circ X; g\circ Y)$ defined by Case 1.
We define $I(X;Y)$ as the supremum of $I(f\circ X; g\circ Y)$ over all finite-range measurable maps 
$f$ and $g$ defined on $\mathcal{X}$ and $\mathcal{Y}$.
This definition is compatible with Case 1 when $\mathcal{X}$ and $\mathcal{Y}$ are finite sets.

\begin{example}
  Let $X$ and $Z$ be real-valued independent random variables.
  Assume that they are Gaussian and obeying 
  \[  X\sim N(a_1, v_1), \quad Z\sim N(a_2, v_2). \]
  Set $Y=X+Z\sim N(a_1+a_2, v_1+v_2)$. Then \cite[Chapter 9, Section 1]{Cover--Thomas}
  \[  I(X;Y) = \frac{1}{2} \log\left(1+\frac{v_1}{v_2}\right). \]
\end{example}

\begin{lemma}[Data-Processing inequality]  \label{lemma: data-processing inequality}
Let $X$ and $Y$ be random variables taking values in measurable spaces $\mathcal{X}$ and $\mathcal{Y}$ respectively.
If $f:\mathcal{Y}\to \mathcal{Z}$ is a measurable map then $I(X; f(Y)) \leq I(X;Y)$.
\end{lemma}

\begin{proof}
This immediately follows from the definition.
\end{proof}

\begin{lemma}  \label{lemma: convergence of mutual information}
Let $\mathcal{X}$ and $\mathcal{Y}$ be finite sets and let $(X_n,Y_n)$ be a sequence of 
random variables taking values in $\mathcal{X}\times \mathcal{Y}$.
If $(X_n, Y_n)$ converges to some $(X, Y)$ in law, then $I(X_n;Y_n)$ converges to $I(X;Y)$.
\end{lemma}

\begin{proof}
This follows from (\ref{eq: definition of mutual information}).
\end{proof}

The next three lemmas are crucial in the proof of Theorem \ref{theorem: dynamical Frostman's lemma}
(dynamical Frostman's lemma).
The proofs are given in \cite[Lemmas 2.7, 2.8, 2.10]{Lindenstrauss--Tsukamoto double VP}.

\begin{lemma}[Subadditivity of mutual information]  \label{lemma: subadditivity of mutual information}
Let $X, Y, Z$ be random variables taking values in finite sets $\mathcal{X}, \mathcal{Y}, \mathcal{Z}$ respectively.
Suppose $X$ and $Y$ are conditionally independent given $Z$. Namely for every $z\in \mathcal{Z}$ with $\mathbb{P}(Z=z)\neq 0$
\[  \mathbb{P}(X=x, Y=y|Z=z) = \mathbb{P}(X=x|Z=z) \mathbb{P}(Y=y|Z=z). \]
Then $I(X, Y;Z) \leq I(X;Z) + I(Y;Z)$.
\end{lemma}

Let $X$ and $Y$ be random variables taking values in finite sets $\mathcal{X}$ and $\mathcal{Y}$.
We set $\mu(x) = \mathbb{P}(X=x)$ and $\nu(y|x) = \mathbb{P}(Y=y|X=x)$, where the latter is defined only for 
$x\in \mathcal{X}$ with $\mathbb{P}(X=x)\neq 0$.
The mutual information $I(X;Y)$ is determined by the distribution of $(X,Y)$, namely $\mu(x)\nu(y|x)$.
So we sometimes write $I(X;Y) = I(\mu, \nu)$.

\begin{lemma}[Concavity/convexity of mutual information]  \label{lemma: convexity of mutual information}
In this notation, $I(\mu, \nu)$ is a concave function of $\mu(x)$ and a convex function of $\nu(y|x)$.
Namely for $0\leq t\leq 1$
\begin{equation*}
   \begin{split}
     I\left((1-t)\mu_1+ t \mu_2, \nu\right) & \geq (1-t)I(\mu_1, \nu) + t I(\mu_2, \nu), \\
     I\left(\mu, (1-t)\nu_1+ t \nu_2\right)  &  \leq (1-t) I(\mu, \nu_1) + t I(\mu, \nu_2).   
   \end{split}
\end{equation*}
\end{lemma}

The following lemma is a key to connect geometric measure theory to rate distortion theory.
We learned this from \cite[Proposition 3.2]{Kawabata--Dembo}.

\begin{lemma}  \label{lemma: geometry and mutual information}
Let $\varepsilon$ and $\delta$ be positive numbers with $2\varepsilon \log(1/\varepsilon) \leq \delta$.
Let $0\leq \tau\leq \min(\varepsilon/3,\delta/2)$ and $s\geq 0$.
Let $(\mathcal{X}, \mathbf{d})$ be a compact metric space with a Borel probability measure $\mu$ satisfying 
\[  \mu(E) \leq \left(\tau + \diam E\right)^s, \quad \forall E\subset \mathcal{X} \text{ with } \diam E < \delta. \]
Let $X$ and $Y$ be random variables taking values in $\mathcal{X}$ with $\mathrm{Law}(X) = \mu$ and 
$\mathbb{E} \mathbf{d}(X, Y) < \varepsilon$.
Then 
\[  I(X;Y) \geq s \log(1/\varepsilon) - K(s+1). \]
Here $K$ is a universal positive constant independent of $\varepsilon, \delta, \tau, s, (\mathcal{X}, \mathbf{d}), \mu$.
\end{lemma}

\subsection{Rate distortion theory}  \label{subsection: rate distortion theory}

We review the definition of rate distortion function here.
See \cite{Shannon, Shannon59} and \cite[Chapter 10]{Cover--Thomas} for more backgrounds.
For a stationary stochastic process $X_1, X_2, X_3, \dots$, its entropy 
\[  H\left(\{X_n\}\right) := \lim_{n\to \infty} \frac{H(X_1,X_2,\dots, X_n)}{n} \]
is equal to the expected number of bits per symbol for describing the process.
Therefore we can say that the Shannon entropy is the fundamental limit of \textit{lossless} data compression.
However if $X_n$ take \textit{continuously} many values, then the entropy is simply infinite.
Namely we cannot describe continuous variables perfectly within finitely many bits.
In this case, we have to consider \textit{lossy} data compression method achieving some \textit{distortion} constraint.
This is the primary object of rate distortion theory.
Rate distortion function is the fundamental limit of data compression in this theory.

Let $(\mathcal{X},T)$ be a dynamical system with a metric $\mathbf{d}$ and an invariant probability measure $\mu$.
For $\varepsilon >0$ we define the \textbf{rate distortion function} $R(\mathbf{d},\mu,\varepsilon)$ as the infimum of 
\[  \frac{I(X;Y)}{N}, \]
where $N$ runs over natural numbers, $X$ and $Y=(Y_0,\dots, Y_{N-1})$ are random variables defined on some 
probability space $(\Omega, \mathbb{P})$ such that all $X$ and $Y_n$ take values in $\mathcal{X}$ and satisfy 
\[  \mathrm{Law}(X) = \mu, \quad \mathbb{E}\left(\frac{1}{N} \sum_{n=0}^{N-1}\mathbf{d}(T^n X, Y_n)\right) < \varepsilon. \]
We define the upper and lower rate distortion dimensions by (\ref{eq: rate distortion dimension}) in 
\S \ref{subsection: backgrounds}.

The rate distortion function $R(\mathbf{d},\mu,\varepsilon)$
is the minimum rate when we quantize the process $\{T^n X\}_{n\in \mathbb{Z}}$ within the average distortion bounded by 
$\varepsilon$ with respect to $\mathbf{d}$.
See \cite[Chapter 11]{Gray} and \cite{LDN, ECG} for the precise meaning of this statement.

\begin{remark} \label{remark: rate distortion function}
In the above definition we can restrict $Y$ to be a finite-range random variable.
(``Finite range'' means that its distribution is supported in a finite subset of $\mathcal{X}$.)
Indeed, take a finite partition $\mathcal{P}$ of $\mathcal{X}$ and pick $x_P\in P$ for each $P\in \mathcal{P}$.
Define $f:\mathcal{X} \to \mathcal{X}$ by $f(P) = \{x_P\}$.
Set $Z= (Z_0, \dots, Z_{N-1}) = \left(f(Y_0), \dots, f(Y_{N-1})\right)$.
If $\mathcal{P}$ is sufficiently fine then 
\[ \mathbb{E}\left(\frac{1}{N}\sum_{n=0}^{N-1} \mathbf{d}(T^n X, Z_n) \right) < \varepsilon. \]
 On the other hand, from the data-processing inequality (Lemma \ref{lemma: data-processing inequality})
\[   I(X;Z) \leq I(X;Y). \]
The random variable $Z$ takes only finitely many values even if $Y$ does not.
\end{remark}

\section{Mean Hausdorff dimension with potential bounds mean dimension with potential: 
proofs of Theorem \ref{theorem: mean Hausdorff dimension bounds mean dimension} and Proposition 
\ref{prop: metric mean dimension bounds rate distortion dimension}}
\label{section: mean Hausdorff dimension bounds mean dimension}

Here we prove Theorem \ref{theorem: mean Hausdorff dimension bounds mean dimension} and 
Proposition \ref{prop: metric mean dimension bounds rate distortion dimension}.
The main issue is to prove that mean Hausdorff dimension with potential 
bounds mean dimension with potential.
The rest of the statements are easy.

\subsection{Proof of Proposition \ref{prop: metric mean dimension bounds rate distortion dimension}}
\label{subsection: proof of metric mean dimension bounds rate distortion dimension}

\begin{lemma}  \label{lemma: KL divergence is nonnegative}
Let $a_1,\dots,a_n$ be real numbers and $\mathbf{p} = (p_1,\dots, p_n)$ a probability vector.
For $\varepsilon >0$
\[  \sum_{i=1}^n \left(-p_i \log p_i + p_i a_i \log(1/\varepsilon) \right) \leq 
     \log \left(\sum_{i=1}^n (1/\varepsilon)^{a_i}\right). \]
\end{lemma}

\begin{proof}
We can prove this by a simple calculus \cite[p. 217, Lemma 9.9]{Walter book}.
Instead of giving it, we briefly describe the information theoretic meaning of the above inequality.
This is more instructive.
Consider a probability vector 
\[ \mathbf{q} = (q_1, \dots, q_n) :=  \frac{1}{\sum_{i=1}^n (1/\varepsilon)^{a_i}} 
   \left((1/\varepsilon)^{a_1}, \dots, (1/\varepsilon)^{a_n}\right). \]
The Kullback--Leibler distance $D(\mathbf{p}|| \mathbf{q})$ is always nonnegative
\cite[Theorem 2.6.3]{Cover--Thomas}:
\[  D(\mathbf{p}|| \mathbf{q}) : = \sum_{i=1}^n p_i \log \frac{p_i}{q_i} \geq 0. \]
Expanding this inequality, we get the above statement.
\end{proof}

\begin{proposition}[$=$ Proposition \ref{prop: metric mean dimension bounds rate distortion dimension}]
\label{prop: metric mean dimension bounds rate distortion dimension restated}

Let $(\mathcal{X}, T)$ be a dynamical system with a metric $\mathbf{d}$ and an invariant probability measure $\mu$.
Let $\varphi:\mathcal{X}\to \mathbb{R}$ be a continuous function. Then
\begin{equation*}
  \begin{split}
    \overline{\rdim}(\mathcal{X},T,\mathbf{d},\mu) + \int_{\mathcal{X}} \varphi \, d\mu  & 
    \leq \overline{\mdim}_{\mathrm{M}}(\mathcal{X}, T, \mathbf{d},  \varphi), \\
    \underline{\rdim}(\mathcal{X},T,\mathbf{d}, \mu) + \int_{\mathcal{X}} \varphi \, d\mu  &
         \leq \underline{\mdim}_{\mathrm{M}}(\mathcal{X},T, \mathbf{d}, \varphi).    
  \end{split}
\end{equation*}
\end{proposition}

\begin{proof}
Let $X$ be a random variable taking values in $\mathcal{X}$ and obeying $\mu$.
Let $N>0$ and let $\mathcal{X} = U_1\cup \dots \cup U_n$ be an open cover with 
$\diam (U_i, \mathbf{d}_N) < \varepsilon$ for all $i$.
Pick $x_i\in U_i$.
We define a random variable $Y$ by 
\[  Y= (x_i, T x_i, \dots, T^{N-1} x_i) \quad \text{if } X\in U_i\setminus (U_1\cup \dots \cup U_{i-1}). \] 
Obviously 
\[  \frac{1}{N}\sum_{k=0}^{N-1}\mathbb{E} \mathbf{d}(T^k X, Y_k) < \varepsilon. \]
Set $p_i = \mu\left(U_i\setminus (U_1\cup \dots\cup U_{i-1})\right)$.
Then 
\[  I(X;Y) \leq H(Y) \leq - \sum_{i=1}^n p_i\log p_i. \]
Set $a_i = \sup_{U_i} S_N\varphi$.
Then 
\[ N \int_{\mathcal{X}} \varphi \, d\mu = \int_{\mathcal{X}} S_N \varphi \, d\mu \leq \sum_{i=1}^n p_i a_i. \]
Hence 
\begin{equation*}
  \begin{split}
     R(\mathbf{d},\mu,\varepsilon) + \log (1/\varepsilon) \int_{\mathcal{X}} \varphi \, d\mu & \leq 
    \frac{I(X;Y)}{N} + \log(1/\varepsilon) \int_{\mathcal{X}} \varphi\, d\mu \\
     & \leq \frac{1}{N} \sum_{i=1}^n \left(-p_i \log p_i + p_i a_i \log(1/\varepsilon) \right) \\
    &\leq \frac{1}{N}    \log \left(\sum_{i=1}^n (1/\varepsilon)^{a_i}\right)   \quad 
    \text{by Lemma \ref{lemma: KL divergence is nonnegative}}.
  \end{split}
\end{equation*}
Thus 
\[  R(\mathbf{d},\mu,\varepsilon) + \log(1/\varepsilon) \int_{\mathcal{X}}\varphi \, d\mu \leq 
     \frac{\log \#(\mathcal{X},\mathbf{d}_N,S_N\varphi,\varepsilon)}{N}. \]
Let $N\to \infty$. Then 
\[   R(\mathbf{d},\mu,\varepsilon) + \log (1/\varepsilon) \int_{\mathcal{X}} \varphi \, d\mu 
      \leq P(\mathcal{X}, T, \mathbf{d}, \varphi, \varepsilon). \]
Divide this by $\log (1/\varepsilon)$ and take the limit of $\varepsilon \to 0$.
\end{proof}

\subsection{Remarks on width dimension with potential}  \label{subsection: remarks on width dimension with potential}

The proof of Theorem \ref{theorem: mean Hausdorff dimension bounds mean dimension} basically
follows the methods developed in \cite[Theorem 4.2]{Lindenstrauss--Weiss} and 
\cite[Proposition 3.2]{Lindenstrauss--Tsukamoto double VP}.
However there is an additional technical issue around the quantity $\widim_\varepsilon(\mathcal{X}, \mathbf{d}, \varphi)$
introduced in (\ref{eq: widim with potential}).
This subsection is a preparation for it.

Let $P$ be a simplicial complex and $a\in P$.
Recall that we defined the local dimension $\dim_a P$ as the maximum of $\dim \Delta$
where $\Delta\subset P$ is a simplex containing $a$.
We define the \textbf{small local dimension} $\dim'_a P$ as the \textit{minimum} of $\dim \Delta$
where $\Delta\subset P$ is a simplex containing $a$.
See Figure \ref{figure: small local dimension}.

\begin{figure}[h]
    \centering
    \includegraphics[width=3.0in]{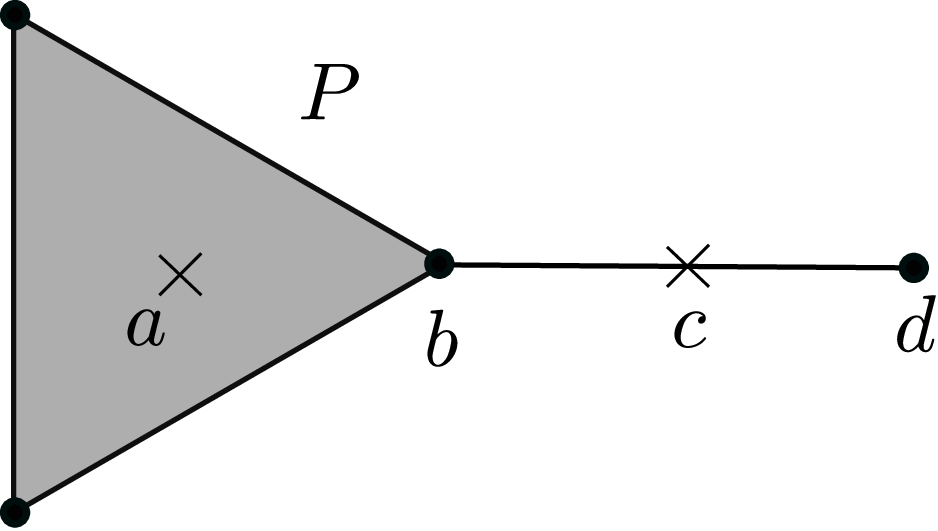}
    \caption{Here $P$ has four vertexes (denoted by dots), four $1$-dimensional simplexes and one $2$-dimensional simplex. 
    The points $b$ and $d$ are vertexes of $P$ wheres $a$ and $c$ are not. We have
    $\dim'_a P =2$,  $\dim'_b P =0$, $\dim'_c P =1$ and $\dim'_d P =0$.
    Recall $\dim_a P = \dim_b P =2$ and $\dim_c P = \dim_d P =1$.}   \label{figure: small local dimension}
\end{figure}

\begin{remark}   \label{remark: small local dimension is combinatorial}
The local dimension $\dim_a P$ is a topological quantity in the sense that 
$\dim_a P$ is equal to the minimum of the topological dimension $\dim U$ where $U\subset P$ is
a neighborhood of $a$.
The small local dimension $\dim'_a P$ is a \textit{combinatorial} quantity.
It depends on the 
combinatorial structure of $P$. 
In particular if $Q$ is a subdivision of $P$ then 
$\dim'_a Q \leq \dim'_a P$.
\end{remark}

Let $(\mathcal{X}, \mathbf{d})$ be a compact metric space with a continuous function 
$\varphi:\mathcal{X}\to \mathbb{R}$.
For $\varepsilon>0$ we set 
\begin{equation*} 
   \begin{split}
     & \widim'_\varepsilon(\mathcal{X}, \mathbf{d}, \varphi)  \\ 
     & =  \inf\left\{\sup_{x\in \mathcal{X}} \left(\dim'_{f(x)} P + \varphi(x)\right) \middle|
    \parbox{3in}{\centering $P$ is a simplicial complex and $f:\mathcal{X}\to P$ is an $\varepsilon$-embedding}\right\}.  
   \end{split}
\end{equation*}    
It follows from Remark \ref{remark: small local dimension is combinatorial} that 
this is also given by
\begin{equation*} 
   \begin{split}
     & \widim'_\varepsilon(\mathcal{X}, \mathbf{d}, \varphi)  \\ 
     & =  \inf\left\{\sup_{x\in \mathcal{X}} \left(\dim'_{f(x)} P + \varphi(x)\right) \middle|
    \parbox{3in}{\centering $P$ is a simplicial complex and $f:\mathcal{X}\to P$ is a continuous map satisfying 
    $\diam f^{-1}\left(\Delta\right) < \varepsilon$ for all simplexes $\Delta \subset P$}\right\}.  
   \end{split}
\end{equation*}    

We set 
\[  \mathrm{var}_\varepsilon(\varphi, \mathbf{d}) = 
    \sup\{|\varphi(x)-\varphi(y)| \, |\, \mathbf{d}(x,y) < \varepsilon\}. \]
For a simplicial complex $P$, we say that a continuous map $f:\mathcal{X}\to P$ is \textbf{essential}
if no proper subcomplex of $P$ contains $f(\mathcal{X})$.

\begin{lemma} \label{lemma: widim and small widim}
\[  \widim'_\varepsilon(\mathcal{X}, \mathbf{d}, \varphi) 
     \leq \widim_\varepsilon(\mathcal{X}, \mathbf{d}, \varphi) 
     \leq \widim'_\varepsilon(\mathcal{X}, \mathbf{d}, \varphi) + \var_\varepsilon(\varphi, \mathbf{d}). \]
\end{lemma}

\begin{proof}
$\widim'_\varepsilon(\mathcal{X}, \mathbf{d}, \varphi) \leq \widim_\varepsilon(\mathcal{X}, \mathbf{d}, \varphi)$
is obvious.
Let $s>\widim'_\varepsilon(\mathcal{X}, \mathbf{d}, \varphi)$.
There are a simplicial complex $P$ and a continuous map 
$f:\mathcal{X}\to P$ such that $\diam f^{-1}(\Delta) < \varepsilon$ for all 
simplexes $\Delta \subset P$ and that 
\[  \dim'_{f(x)} P +\varphi(x) < s \quad (\forall x\in \mathcal{X}). \]
Moreover we can assume that $f$ is essential.

Let $x\in \mathcal{X}$ and let $\Delta\subset P$ be a maximum simplex containing $f(x)$.
Since $f$ is essential, there is $y\in \mathcal{X}$ such that $f(y)$ is an interior point of $\Delta$.
Then 
\[  \dim'_{f(y)} P = \dim \Delta = \dim_{f(x)} P. \]
We have $\mathbf{d}(x,y) \leq \diam f^{-1}(\Delta) < \varepsilon$.
So $|\varphi(x)-\varphi(y)| \leq \var_\varepsilon(\varphi, \mathbf{d})$.
Hence 
\[ \dim_{f(x)} P + \varphi(x) \leq \dim'_{f(y)} P + \varphi(y) + \var_\varepsilon(\varphi, \mathbf{d}) < 
    s + \var_\varepsilon(\varphi, \mathbf{d}). \]
Since $x\in \mathcal{X}$ is arbitrary, 
$\widim_\varepsilon(\mathcal{X}, \mathbf{d},\varphi) \leq s + \var_\varepsilon(\varphi, \mathbf{d})$.
\end{proof}

\begin{corollary} \label{cor: definition of mean dimension with potential by small local dimension}
Let $T:\mathcal{X}\to \mathcal{X}$ be a homeomorphism. Then 
\[  \mdim(\mathcal{X}, T, \varphi) = 
     \lim_{\varepsilon\to 0} \left(\lim_{N\to \infty} \frac{\widim'_\varepsilon(\mathcal{X}, \mathbf{d}_N, S_N\varphi)}{N}\right). \]
Here $\widim'_\varepsilon(\mathcal{X},\mathbf{d}_N,S_N\varphi)$ is subadditive in $N$ and monotone in $\varepsilon$.     
\end{corollary}

\begin{proof}
Recall that we defined 
\[   \mdim(\mathcal{X}, T, \varphi) = 
     \lim_{\varepsilon\to 0} \left(\lim_{N\to \infty} \frac{\widim_\varepsilon(\mathcal{X}, \mathbf{d}_N, S_N\varphi)}{N}\right). \]
From Lemma \ref{lemma: widim and small widim}
\[   \widim'_\varepsilon(\mathcal{X}, \mathbf{d}_N, S_N\varphi) 
     \leq \widim_\varepsilon(\mathcal{X}, \mathbf{d}_N, S_N\varphi) 
     \leq \widim'_\varepsilon(\mathcal{X}, \mathbf{d}_N, S_N\varphi) + \var_\varepsilon(S_N\varphi, \mathbf{d}_N). \]
Since $\var_\varepsilon(S_N\varphi, \mathbf{d}_N) \leq N \cdot \var_\varepsilon(\varphi, \mathbf{d})$
\[ \lim_{\varepsilon \to 0} \left(\lim_{N\to \infty} \frac{\var_\varepsilon(S_N \varphi, \mathbf{d}_N)}{N}\right) 
    =  \lim_{\varepsilon\to 0} \var_\varepsilon(\varphi,\mathbf{d}) = 0. \] 
\end{proof}

\subsection{Proof of Theorem \ref{theorem: mean Hausdorff dimension bounds mean dimension}}
\label{subsection: proof of mean Hausdorff dimension bound mean dimension}

\begin{theorem}[$=$ Theorem \ref{theorem: mean Hausdorff dimension bounds mean dimension}]
\label{theorem: mean Hausdorff dimension bounds mean dimension restated}
Let $(\mathcal{X}, T)$ be a dynamical system with a metric $\mathbf{d}$ and a continuous function 
$\varphi: \mathcal{X}\to \mathbb{R}$.
\[  \mdim(\mathcal{X},T,\varphi) \leq \mdim_{\mathrm{H}}(\mathcal{X},T,\mathbf{d},\varphi) \leq 
     \underline{\mdim}_{\mathrm{M}}(\mathcal{X},T,\mathbf{d},\varphi). \]
\end{theorem}

\begin{proof}[Proof of Theorem \ref{theorem: mean Hausdorff dimension bounds mean dimension restated} (easy part): proof of 
$\mdim_{\mathrm{H}}(\mathcal{X},T,\mathbf{d},\varphi)\leq \underline{\mdim}_{\mathrm{M}}(\mathcal{X},T,\mathbf{d},\varphi)$]
Here we prove $\mdim_{\mathrm{H}}(\mathcal{X}, T, \mathbf{d}, \varphi) \leq 
\underline{\mdim}_{\mathrm{M}}(\mathcal{X},T,\mathbf{d},\varphi)$.
This is straightforward.
Let $0<\varepsilon<1$ and $N>0$.
Let $\mathcal{X}= U_1 \cup \dots\cup U_n$ be an open cover with 
$\diam (U_i, \mathbf{d}_N) < \varepsilon$.
For $s\geq \max_{\mathcal{X}} S_N\varphi$
\begin{equation*}
   \begin{split}
   \mathcal{H}^s_\varepsilon(\mathcal{X},\mathbf{d}_N, S_N\varphi) & 
   \leq \sum_{i=1}^n \left(\diam(U_i,\mathbf{d}_N)\right)^{s-\sup_{U_i}S_N\varphi} \\
   & \leq \sum_{i=1}^n \varepsilon^{s-\sup_{U_i}S_N\varphi}
   = \varepsilon^s \cdot \sum_{i=1}^n (1/\varepsilon)^{\sup_{U_i} S_N\varphi}.
   \end{split}
\end{equation*}   
So 
\[   \mathcal{H}^s_\varepsilon(\mathcal{X},\mathbf{d}_N, S_N\varphi) 
     \leq \varepsilon^s \cdot \#(\mathcal{X},\mathbf{d}_N,S_N\varphi, \varepsilon). \]
Hence $\mathcal{H}^s_\varepsilon(\mathcal{X},\mathbf{d}_N, S_N\varphi) <1$ 
 for $s> \log\#(\mathcal{X},\mathbf{d}_N,S_N\varphi,\varepsilon)/\log(1/\varepsilon)$.
 (Notice that we always have $\#(\mathcal{X},\mathbf{d}_N,S_N\varphi,\varepsilon)  \geq (1/\varepsilon)^{\max_{\mathcal{X}}S_N\varphi}$
 and hence $\log\#(\mathcal{X},\mathbf{d}_N,S_N\varphi,\varepsilon)/\log(1/\varepsilon) \geq \max_{\mathcal{X}} S_N\varphi$.)
This implies 
\[ \dim_{\mathrm{H}}(\mathcal{X},\mathbf{d}_N,S_N\varphi,\varepsilon) \leq 
    \frac{\log\#(\mathcal{X},\mathbf{d}_N,S_N\varphi,\varepsilon)}{\log(1/\varepsilon)}. \]
Divide this by $N$ and take the limits of $N\to \infty$:
\[  \limsup_{N\to \infty} \left(\frac{\dim_{\mathrm{H}}(\mathcal{X},\mathbf{d}_N,S_N\varphi,\varepsilon)}{N}\right)
      \leq \frac{P(\mathcal{X},T,\mathbf{d},\varphi,\varepsilon)}{\log(1/\varepsilon)}. \]   
Letting $\varepsilon\to 0$, we get $\mdim_{\mathrm{H}}(\mathcal{X}, T, \mathbf{d}, \varphi) \leq 
\underline{\mdim}_{\mathrm{M}}(\mathcal{X},T,\mathbf{d},\varphi)$.
\end{proof}

Next we prove that mean Hausdorff dimension with potential bounds mean dimension with potential.
We need some preparations.

Let $(\mathcal{X}, \mathbf{d})$ be a compact metric space.
For $s\geq 0$, we define 
\[  \mathcal{H}^s_\infty(\mathcal{X}, \mathbf{d}) = 
\inf\left\{ \sum_{i=1}^\infty \left(\diam E_i\right)^{s} \middle|\,
    \mathcal{X} = \bigcup_{i=1}^\infty E_i \right\}. \]
We denote the standard Lebesgue measure on $\mathbb{R}^N$ by $\nu_N$.
We set $\norm{x} = \max_{1\leq i\leq N} |x_i|$ for $x\in \mathbb{R}^N$.
For $A\subset \{1,2,\dots, N\}$ we define $\pi_A:[0,1]^N\to [0,1]^A$ as the projection to the $A$-coordinates.
The next lemma was given in \cite[Claim 3.3]{Lindenstrauss--Tsukamoto double VP}.

\begin{lemma} \label{lemma: measure of projections}
Let $K\subset [0,1]^N$ be a closed subset and $0\leq n\leq N$.  
    \begin{enumerate}
        \item $\nu_N(K) \leq 2^N \mathcal{H}^N_\infty\left(K, \norm{\cdot}\right)$.
        \item $\nu_N\left(\bigcup_{|A|\geq n} \pi_A^{-1}(\pi_A K)\right) \leq 4^N \mathcal{H}^n_\infty\left(K, \norm{\cdot}\right)$.
    \end{enumerate}
\end{lemma}

\begin{proof}
(1) Let $K=\bigcup_{i=1}^\infty E_i$ with $l_i := \diam \left(E_i, \norm{\cdot}\right)$.
Take $x_i \in E_i$.
We have $E_i\subset x_i+[-l_i, l_i]^N$. Hence 
\[ \nu_N(K) \leq \sum_{i=1}^\infty (2l_i)^N = 2^N \sum_{i=1}^\infty l_i^N. \]

(2) The volume $\nu_N\left(\bigcup_{|A|\geq n} \pi_A^{-1}(\pi_A K)\right)$ is bounded by 
\[  \sum_{|A|\geq n} \nu_N\left(\pi_A^{-1}(\pi_A K)\right) = \sum_{|A|\geq n} \nu_{|A|} (\pi_A K). \]
We apply the above (1) to $\pi_A K\subset [0,1]^A$. Then 
\[  \nu_{|A|}(\pi_A K) \leq 2^{|A|} \mathcal{H}^{|A|}_\infty \left(\pi_A K, \norm{\cdot}\right) 
     \leq 2^N \mathcal{H}_\infty^{|A|}\left(\pi_A K, \norm{\cdot}\right) 
     \leq 2^N \mathcal{H}^{|A|}_\infty \left(K, \norm{\cdot}\right). \]
\begin{equation*}
    \nu_N\left(\bigcup_{|A|\geq n} \pi_A^{-1}(\pi_A K)\right) \leq 2^N \sum_{|A|\geq n} \mathcal{H}^{|A|}_\infty\left(K, \norm{\cdot}\right)
    \leq 2^N \sum_{|A|\geq n} \mathcal{H}^n_\infty\left(K, \norm{\cdot}\right) \leq 4^N \mathcal{H}^n_\infty\left(K, \norm{\cdot}\right).
\end{equation*}   
\end{proof}

The following lemma is the central ingredient of the proof of Theorem
\ref{theorem: mean Hausdorff dimension bounds mean dimension restated}.
The proof uses the method originally introduced in \cite[Theorem 4.2]{Lindenstrauss--Weiss}.
But our case is a bit more involved because we need to control the information of a potential function.

\begin{lemma} \label{lemma: widim estimate via hausdorff measure}
Let $(\mathcal{X}, \mathbf{d})$ be a compact metric space with a continuous function 
$\varphi:\mathcal{X}\to \mathbb{R}$.
Let $\varepsilon>0, L>0$ and $s\geq \max_{\mathcal{X}}\varphi$ be real numbers.
Suppose there exists a Lipschitz map $f:\mathcal{X}\to [0,1]^N$ such that 
\begin{itemize}
   \item $\norm{f(x)-f(y)} \leq L\cdot \mathbf{d}(x,y)$,
   \item $\norm{f(x)-f(y)} = 1$ if $\mathbf{d}(x,y) \geq \varepsilon$.
\end{itemize}      
Moreover, suppose 
\begin{equation}  \label{eq: assumption on 1-hausdorff measure}
      4^N  (L+1)^{1+s+\norm{\varphi}_\infty}  \mathcal{H}^s_1\left(\mathcal{X},\mathbf{d},\varphi\right) < 1, 
\end{equation}      
where $\norm{\varphi}_\infty = \max_{\mathcal{X}} |\varphi|$.
Then 
\[   \widim'_\varepsilon\left(\mathcal{X},\mathbf{d},\varphi\right) \leq s+1. \]
\end{lemma}

\begin{proof}
Set $m = \lfloor s+\norm{\varphi}_\infty \rfloor$.
For integers $0\leq k\leq m$ we set 
\[ \mathcal{X}_k = \{x\in \mathcal{X}|\, k\leq s - \varphi(x)\leq k+1\}. \]
We have $\mathcal{X}= \bigcup_{k=0}^{m} \mathcal{X}_k$ and 
\[  \mathcal{H}^{k+1}_\infty \left(\mathcal{X}_k, \mathbf{d}\right) \leq 
     \mathcal{H}^s_1\left(\mathcal{X}, \mathbf{d},\varphi\right). \]
Then 
\[  \mathcal{H}^{k+1}_\infty\left(f(\mathcal{X}_k),\norm{\cdot}\right) \leq 
    L^{k+1} \mathcal{H}^{k+1}_\infty \left(\mathcal{X}_k, \mathbf{d}\right) \leq 
    L^{k+1} \mathcal{H}^s_1\left(\mathcal{X}, \mathbf{d},\varphi\right). \]
By Lemma \ref{lemma: measure of projections} (2)
\begin{equation*}
   \begin{split}
   \nu_N\left(\bigcup_{k=0}^{m} \bigcup_{|A|\geq k+1} \pi_A^{-1}\left(\pi_A(f\mathcal{X}_k)\right)\right) &\leq
   \sum_{k=0}^{m} 4^N \mathcal{H}^{k+1}_\infty\left(f(\mathcal{X}_k), \norm{\cdot}\right)  \\
   &\leq 4^N \left(L+L^2+\dots+L^{m+1}\right) \mathcal{H}^s_1(\mathcal{X},\mathbf{d},\varphi) \\
   & \leq    4^N (L+1)^{m+1}    \mathcal{H}^s_1\left(\mathcal{X}, \mathbf{d},\varphi\right)  \\
   & < 1  \quad \text{by (\ref{eq: assumption on 1-hausdorff measure})}.
   \end{split}
\end{equation*}   
Then we can find a point $q\in (0,1)^N$ 
outside of the set $\bigcup_{k=0}^{m} \bigcup_{|A|\geq k+1} \pi_A^{-1}\left(\pi_A(f\mathcal{X}_k)\right)$.
We have 
\[  f(\mathcal{X}_k) \cap \bigcup_{|A|\geq k+1} \pi_A^{-1}(\pi_A(q)) = \emptyset \]
for all $0\leq k\leq m$.
(For $k\geq N$ the set $\bigcup_{|A|\geq k+1} \pi_A^{-1}(\pi_A(q))$ is empty.)
 
Let $P_n \subset [0,1]^N$ $(0\leq n\leq N)$ be the $n$-skeleton, namely the set of 
$x\in [0,1]^N$ satisfying $|\{i|\, x_i = 0 \text{ or } 1\}| \geq N-n$.
We set 
\[  C_n = P_n \cap \bigcup_{|A|=n} \pi_A^{-1}(\pi_A(q)). \]
Each facet of $P_n$ contains exactly one point of $C_n$ (in particular, $C_n$ is a finite set). 
We define a continuous map $g_n:P_n\setminus C_n\to P_{n-1}$ by the central projection from each point of $C_n$.
For $1\leq l < n$ 
\[  g_n\left(P_n \setminus \bigcup_{|A|=l} \pi_A^{-1}(\pi_A(q))\right) = P_{n-1}\setminus \bigcup_{|A|=l} \pi_A^{-1}(\pi_A(q)). \]
For $0\leq t\leq 1$ we set 
\[  g_{n, t}(x) = (1-t) g_n(x) + t x \quad (x\in P_n\setminus C_n). \]
Let $1\leq n, n'\leq N$ and let
\[ x\in [0,1]^N \setminus \bigcup_{|A|\geq n} \pi_A^{-1}(\pi_A(q)), \quad 
   x' \in [0,1]^N \setminus \bigcup_{|A|\geq n'} \pi_A^{-1}(\pi_A(q)). \]
If $\norm{x-x'} = 1$ then for any $0\leq t, t'\leq 1$
\[ \norm{g_{n,t}\circ g_{n+1}\circ \dots \circ g_N(x) - g_{n',t'}\circ g_{n'+1}\circ \dots \circ g_N(x')}=1. \]
This is because if $x_i=0 \text{ or } 1$ then the $i$-th coordinate of $g_{n,t}(x)$ is equal to $x_i$.

We define a continuous map $h:\mathcal{X}\to [0,1]^N$ as follows.
Let $x\in \mathcal{X}_k$ $(0\leq k\leq m)$. 
If $k+1>N$ then we set $h(x) = x$. Otherwise we set 
\[ h(x) = g_{k+1, s-\varphi(x)-k}\circ g_{k+2} \circ \dots \circ g_N \circ f (x). \]
This definition is compatible on $\mathcal{X}_k\cap \mathcal{X}_{k-1}$.
The map $h$ is an $\varepsilon$-embedding.
For $x\in \mathcal{X}_k$ the point $h(x)$ belongs to $P_{\min(k+1, N)}$.
Note $k+1\leq s-\varphi(x) +1$.
We can introduce a simplicial complex structure on $[0,1]^N$ such that $P_n$ is an 
$n$-dimensional subcomplex.
Then for $x\in  \mathcal{X}_k$
\[  \dim'_{h(x)} [0,1]^N \leq \min(k+1, N) \leq s-\varphi(x)+1. \]
Therefore
\[ \widim'_\varepsilon(\mathcal{X}, \mathbf{d},\varphi) \leq \sup_{x\in \mathcal{X}} \left(s-\varphi(x)+1 + \varphi(x)\right) = s+1. \]
\end{proof}

Now we are ready to prove the main part of Theorem \ref{theorem: mean Hausdorff dimension bounds mean dimension restated}.

\begin{proof}[Proof of Theorem \ref{theorem: mean Hausdorff dimension bounds mean dimension restated} (main part): 
proof of $\mdim(\mathcal{X},T,\mathbf{d},\varphi) \leq \mdim_{\mathrm{H}}(\mathcal{X},T,\mathbf{d},\varphi)$]

Given $\varepsilon >0$, we take a Lipschitz map $f:\mathcal{X}\to [0,1]^M$ for some $M$ such that 
\[  \mathbf{d}(x,y) \geq \varepsilon \Longrightarrow \norm{f(x)-f(y)} = 1. \]
(Such $f$ can be constructed by using appropriate bump functions.)
Let $L>0$ be a Lipschitz constant of $f$, namely $\norm{f(x)-f(y)}\leq L\cdot \mathbf{d}(x,y)$.
For $N>0$ we define $f_N:\mathcal{X}\to [0,1]^{MN}$ by 
\[ f_N(x) = \left(f(x), f(Tx), \dots, f(T^{N-1}x)\right). \]
This satisfies 
\begin{itemize}
  \item $\norm{f_N(x)-f_N(y)} \leq L \cdot \mathbf{d}_N(x,y)$, 
  \item $\norm{f_N(x)-f_N(y)}= 1$ if $\mathbf{d}_N(x,y) \geq \varepsilon$.
\end{itemize}

Take $s>\mdim_{\mathrm{H}}(\mathcal{X}, T,\mathbf{d},\varphi)$.
Let $\tau>0$ be arbitrary.
Choose $0<\delta<1$ satisfying 
\begin{equation} \label{eq: choice of delta in the proof of mdim leq mdim_H}
     4^M \cdot (L+1)^{1+s+\tau +\norm{\varphi}_\infty} \cdot \delta^\tau < 1. 
\end{equation}     
From $\mdim_{\mathrm{H}}(\mathcal{X},T,\mathbf{d},\varphi) < s$, we can find
$0<N_1<N_2<N_3<\dots\to \infty$ satisfying 
$\dim_{\mathrm{H}}(\mathcal{X},\mathbf{d}_{N_i}, S_{N_i}\varphi,\delta) < s N_i$.
(Notice that we actually use only $\underline{\mdim}_{\mathrm{H}}(\mathcal{X}, T,\mathbf{d},\varphi) < s$ here.)
Then $\mathcal{H}_\delta^{s N_i}\left(\mathcal{X},\mathbf{d}_{N_i},S_{N_i}\varphi \right) <1$ and hence 
\[ \mathcal{H}_\delta^{(s+\tau)N_i}\left(\mathcal{X},\mathbf{d}_{N_i}, S_{N_i}\varphi \right) 
   \leq \delta^{\tau N_i} \mathcal{H}_\delta^{s N_i}\left(\mathcal{X},\mathbf{d}_{N_i}, S_{N_i}\varphi \right) < \delta^{\tau N_i}. \]
By (\ref{eq: choice of delta in the proof of mdim leq mdim_H})
\begin{equation*}
   \begin{split}
      4^{M N_i}\cdot (L+1)^{1 + (s+\tau)N_i + \norm{S_{N_i}\varphi}_\infty} 
    \mathcal{H}_1^{(s+\tau)N_i}\left(\mathcal{X}, \mathbf{d}_{N_i}, S_{N_i}\varphi\right) 
    &< \left\{4^M \cdot (L+1)^{1+s+\tau + \norm{\varphi}_\infty}\cdot \delta^\tau \right\}^{N_i}  \\
    &<  1.
    \end{split}
\end{equation*}    
Now we apply Lemma \ref{lemma: widim estimate via hausdorff measure} to 
$(\mathcal{X},\mathbf{d}_{N_i}, S_{N_i}\varphi)$ and 
$f_{N_i}:\mathcal{X}\to [0,1]^{MN_i}$.
(We replace the parameter $s$ in the statement of Lemma \ref{lemma: widim estimate via hausdorff measure} with $(s+\tau)N_i$.)
Then 
\[  \widim'_\varepsilon\left(\mathcal{X},\mathbf{d}_{N_i}, S_{N_i}\varphi\right) \leq (s+\tau)N_i + 1. \]
Hence 
\[  \lim_{N\to \infty} \frac{\widim'_\varepsilon\left(\mathcal{X},\mathbf{d}_{N}, S_{N}\varphi\right)}{N} \leq s + \tau. \]
Let $s\to \mdim_{\mathrm{H}}(\mathcal{X},T,\mathbf{d},\varphi)$, $\tau\to 0$ and $\varepsilon \to 0$:
\[  \lim_{\varepsilon\to 0} \left(\lim_{N\to \infty} \frac{\widim'_\varepsilon\left(\mathcal{X},\mathbf{d}_{N}, S_{N}\varphi\right)}{N}\right)
   \leq \mdim_{\mathrm{H}}(\mathcal{X},T,\mathbf{d},\varphi). \]
By Corollary \ref{cor: definition of mean dimension with potential by small local dimension}, this proves 
$\mdim(\mathcal{X},T,\varphi) \leq  \mdim_{\mathrm{H}}(\mathcal{X},T,\mathbf{d},\varphi)$.
\end{proof}

\begin{remark}
The above proof actually shows 
$\mdim(\mathcal{X},T,\varphi) \leq \underline{\mdim}_{\mathrm{H}}(\mathcal{X},T,\mathbf{d},\varphi)$.
\end{remark}

\section{Dynamical Frostman's lemma: proofs of Theorem \ref{theorem: dynamical Frostman's lemma} and 
Corollary \ref{corollary: mean dimension, rate distortion dimension and metric mean dimension}}
\label{section: dynamical Frostman's lemma}

Here we prove Theorem \ref{theorem: dynamical Frostman's lemma} (a version of dynamical Frostman's lemma) and Corollary 
\ref{corollary: mean dimension, rate distortion dimension and metric mean dimension}.
The first two subsections are preparations.

\subsection{Tame growth of covering numbers}  \label{subsection: tame growth of covering numbers}

Let $(\mathcal{X},\mathbf{d})$ be a compact metric space.
For $\varepsilon>0$ we define $\#(\mathcal{X},\mathbf{d},\varepsilon)$ as the minimum cardinarity of 
open covers $\mathcal{U}$ of $\mathcal{X}$ satisfying $\diam \, U < \varepsilon$ for all $U\in \mathcal{U}$.
This is a special case of the quantity introduced in (\ref{eq: covering number with potential})
in \S \ref{subsection: main ingredients of the proof}.
Namely we have 
$\#(\mathcal{X},\mathbf{d},\varepsilon) = \#(\mathcal{X},\mathbf{d},0, \varepsilon)$.

\begin{definition}  \label{def: tame growth of covering numbers}
A compact metric space $(\mathcal{X},\mathbf{d})$ is said to have the \textbf{tame growth of covering numbers} if 
for any $\delta>0$ 
\[  \lim_{\varepsilon \to 0} \varepsilon^\delta \log \#(\mathcal{X},\mathbf{d},\varepsilon) = 0. \] 
\end{definition}

For example, the Euclidean metric on any compact subset of $\mathbb{R}^N$ has the tame growth of covering numbers.
The metric $\rho$ on $[0,1]^\mathbb{Z}$ defined by 
\begin{equation*}
    \rho(x,y) = \sum_{n\in\mathbb{Z}} 2^{-|n|} |x_n-y_n|
\end{equation*}
also satisfies the condition.

Indeed the tame growth of covering numbers is a fairly mild condition \cite[Lemma 3.10]{Lindenstrauss--Tsukamoto double VP}:

\begin{lemma} \label{lemma: tame growth is mild}
Let $(\mathcal{X},\mathbf{d})$ be a compact metric space.
There exists a metric $\mathbf{d}'$ on $\mathcal{X}$ (compatible with the topology) 
such that $\mathbf{d}'(x,y) \leq \mathbf{d}(x,y)$ and 
that $(\mathcal{X}, \mathbf{d}')$ has the tame growth of covering numbers.
In particular every compact metrizable space admits a metric having the tame growth of covering numbers.
\end{lemma}

\begin{proof}
Take a countable dense subset $\{x_n\}_{n=1}^\infty$ in $\mathcal{X}$.
We define 
\[ \mathbf{d}'(x,y) = \sum_{n=1}^\infty 2^{-n} \left|\mathbf{d}(x,x_n)-\mathbf{d}(y,x_n)\right|. \]
It is easy to check that this satisfies the requirements.
\end{proof}

\subsection{$L^1$-mean Hausdorff dimension with potential}
\label{subsection: L^1-mean Hausdorff dimension with potential}

Let $(\mathcal{X},T)$ be a dynamical system with a metric $\mathbf{d}$.
For $N\geq 1$ we define a new metric $\overline{\mathbf{d}}_N$ on $\mathcal{X}$ by
\[ \overline{\mathbf{d}}_N(x,y)  = \frac{1}{N} \sum_{n=0}^{N-1} \mathbf{d}(T^n x, T^n y). \]
We are interested in this metric because it is closely related to the distortion condition 
\[  \mathbb{E}\left(\frac{1}{N} \sum_{n=0}^{N-1} \mathbf{d}(T^n X, Y_n)\right) < \varepsilon \]
used in the definition of rate distortion function (\S \ref{subsection: rate distortion theory}).

Let $\varphi:\mathcal{X}\to \mathbb{R}$ be a continuous function.
We define the \textbf{$L^1$-mean Hausdorff dimension with potential} by 
\[ \mdim_{\mathrm{H}, L^1}(\mathcal{X},T,\mathbf{d},\varphi) = 
    \lim_{\varepsilon\to 0} \left(\limsup_{N\to \infty} 
    \frac{\dim_{\mathrm{H}}(\mathcal{X},\overline{\mathbf{d}}_N, S_N\varphi, \varepsilon)}{N}\right). \]
Since $\overline{\mathbf{d}}_N\leq \mathbf{d}_N$, we always have 
\[  \mdim_{\mathrm{H},L^1}(\mathcal{X},T,\mathbf{d},\varphi)
     \leq   \mdim_{\mathrm{H}}(\mathcal{X},T,\mathbf{d},\varphi). \]

\begin{lemma}  \label{lemma: L^1 mean hausdorff dimension is equal to mean hausdorff dimension}
If $(\mathcal{X},\mathbf{d})$ has the tame growth of covering numbers then 
\[  \mdim_{\mathrm{H},L^1}(\mathcal{X},T,\mathbf{d},\varphi)
      =   \mdim_{\mathrm{H}}(\mathcal{X},T,\mathbf{d},\varphi). \]
\end{lemma}

\begin{proof}
It is enough to prove $\mdim_{\mathrm{H}}(\mathcal{X},T,\mathbf{d},\varphi)  \leq
    \mdim_{\mathrm{H},L^1}(\mathcal{X},T,\mathbf{d},\varphi)$.
We use the notations $[N] := \{0,1,2,\dots, N-1\}$ and $\mathbf{d}_A(x,y) := \max_{a\in A} \mathbf{d}(T^a x, T^a y)$
for $A\subset [N]$.

Let $0<\delta<1/2$ and $s>\mathrm{\mdim}_{\mathrm{H},L^1}(\mathcal{X},T,\mathbf{d},\varphi)$ be arbitrary.
For each $\tau>0$ we choose an open cover 
$\mathcal{X}= W_1^\tau\cup \dots \cup W_{M(\tau)}^\tau$ with 
$\diam\left(W_i^\tau, \mathbf{d}\right) < \tau$ and $M(\tau) = \#(\mathcal{X},\mathbf{d},\tau)$.
From the tame growth condition, we can find $0<\varepsilon_0<1$ such that 
\begin{eqnarray}
  && M(\tau)^{\tau^\delta} < 2 \quad (\forall 0<\tau<\varepsilon_0), \label{eq: choice of varepsilon_0 1}  \\ 
  && 2^{2+\delta +(1+2\delta)(s+\norm{\varphi}_\infty)} \cdot \varepsilon_0^{\delta(1-\delta)} < 1.  \label{eq: choice of varepsilon_0 2}
\end{eqnarray}

Let $0<\varepsilon<\varepsilon_0$ be a sufficiently small number, and let $N$ be a sufficiently large natural number.
Since $\mdim_{\mathrm{H}, L^1}(\mathcal{X},T,\mathbf{d},\varphi) < s$, there exists a covering 
$\mathcal{X} = \bigcup_{n=1}^\infty E_n$ with $\tau_n := \diam (E_n, \overline{\mathbf{d}}_N) < \varepsilon$ satisfying 
\begin{equation}  \label{eq: choice of tau_n in the proof of L^1 mean hausdorff dimension is equal to mean hausdorff dimension}
     \sum_{n=1}^\infty \tau_n^{sN-\sup_{E_n} S_N\varphi} < 1 , \quad (sN \geq \max_{\mathcal{X}} S_N\varphi). 
\end{equation}     
Set $L_n = (1/\tau_n)^\delta$ and pick a point $x_n\in E_n$ for each $n$.
Then every $x\in E_n$ satisfies $\overline{\mathbf{d}}_N(x,x_n) <\tau_n$ and hence 
\[  \left|\{k\in [N] |\, \mathbf{d}(T^k x, T^k y) \geq L_n \tau_n\}\right| \leq \frac{N}{L_n}. \]
So there exists $A\subset [N]$ (depending on $x\in E_n$) such that $|A|\leq N/L_n$ and 
$\mathbf{d}_{[N]\setminus A}(x,x_n) < L_n \tau_n$. 
Thus 
\[  E_n \subset \bigcup_{A\subset [N], |A|\leq N/L_n}  B^\circ_{L_n \tau_n} (x_n,\mathbf{d}_{[N]\setminus A}), \]
where $B^\circ_{L_n \tau_n} (x_n,\mathbf{d}_{[N]\setminus A})$ is the open ball of radius $L_n \tau_n$ around $x_n$
with respect to the metric $\mathbf{d}_{[N]\setminus A}$.

Let $A= \{a_1,\dots, a_r\}$. 
We consider a decomposition
\[  B^\circ_{L_n \tau_n} (x_n,\mathbf{d}_{[N]\setminus A})  = \bigcup_{1\leq i_1,\dots, i_r\leq M(\tau_n)}
    B^\circ_{L_n \tau_n} (x_n,\mathbf{d}_{[N]\setminus A}) \cap T^{-a_1} W_{i_1}^{\tau_n} \cap \dots \cap T^{-a_r} W^{\tau_n}_{i_r}. \]
Then $\mathcal{X}$ is covered by the sets 
\begin{equation} \label{eq: auxiliary covering in the proof of L^1 mean hausdorff is equal to mean hausdorff}
    E_n \cap  B^\circ_{L_n \tau_n} (x_n,\mathbf{d}_{[N]\setminus A}) \cap T^{-a_1} W_{i_1}^{\tau_n} \cap \dots \cap T^{-a_r} W^{\tau_n}_{i_r}, 
\end{equation}    
where $n \geq 1$, $A= \{a_1,\dots, a_r\}\subset [N]$ with $r\leq N/L_n$ and $1\leq i_1,\dots, i_r\leq M(\tau_n)$.
The sets (\ref{eq:  auxiliary covering in the proof of L^1 mean hausdorff is equal to mean hausdorff}) 
have diameter less than or equal to $2L_n \tau_n = 2 \tau_n^{1-\delta} < 2\varepsilon^{1-\delta}$
with respect to the metric $\mathbf{d}_{N}$.

Set $m_N= \min_{\mathcal{X}} S_N\varphi$.
We estimate the quantity 
\[  \mathcal{H}^{sN +2\delta(sN - m_N)+\delta N}_{2\varepsilon^{1-\delta}} (\mathcal{X}, \mathbf{d}_N, S_N \varphi).  \]
This is bounded by 
\begin{equation*}
    \sum_{n=1}^\infty 2^N \cdot M(\tau_n)^{N/L_n} \cdot
     \left(2\tau_n^{1-\delta}\right)^{sN + 2\delta(sN-m_N)+\delta N -\sup_{E_n} S_N\varphi}. 
\end{equation*}     
The factor $2^N$ comes from the choice of $A\subset [N]$.
Since $\tau_n < \varepsilon < \varepsilon_0$
\begin{equation*}
   \begin{split}
       \left(2\tau_n^{1-\delta}\right)^{sN + 2\delta(sN-m_N)+\delta N -\sup_{E_n} S_N\varphi}
      & =   \left(2\tau_n^{1-\delta}\right)^{sN + 2\delta(sN-m_N)-\sup_{E_n} S_N\varphi} \cdot 
           \left(2\tau_n^{1-\delta}\right)^{\delta N}  \\
     & \leq       \left(2\tau_n^{1-\delta}\right)^{sN + 2\delta(sN-m_N)-\sup_{E_n} S_N\varphi} \cdot 
         \left(2^\delta\varepsilon_0^{\delta(1-\delta)}\right)^N.
   \end{split}
\end{equation*}    
The term $\left(2\tau_n^{1-\delta}\right)^{sN + 2\delta(sN-m_N)-\sup_{E_n} S_N\varphi}$ is equal to      
\begin{equation*}
   \underbrace{2^{sN + 2\delta (sN-m_N) -\sup_{E_n} S_N\varphi}}_{(I)} \cdot
   \underbrace{\tau_n^{2\delta(sN-m_N)-\delta\left\{sN+2\delta(sN-m_N)-\sup_{E_n} S_N\varphi\right\}}}_{(II)} 
   \cdot \tau_n^{sN-\sup_{E_n} S_N\varphi}.
\end{equation*}                                            
The factor $(I)$ is bounded by 
\[  2^{sN+2\delta (sN + \norm{\varphi}_\infty N) + \norm{\varphi}_\infty N} 
      = 2^{(1+2\delta)(s+\norm{\varphi}_\infty)N}. \] 
The exponent of the factor $(II)$ is bounded from below (note $0<\tau_n<1$) by 
\[  2\delta(sN-m_N) -\delta\left\{sN+2\delta(sN-m_N)-m_N\right\} 
    = \delta(1-2\delta)(sN-m_N) \geq 0.  \]
Here we have used $sN\geq \max_{\mathcal{X}} S_N\varphi \geq m_N$.
Hence the factor $(II)$ is less than or equal to $1$.
Summing up the above estimates, we get
\[     \left(2\tau_n^{1-\delta}\right)^{sN + 2\delta(sN-m_N)+\delta N -\sup_{E_n} S_N\varphi}
    \leq 2^{(1+2\delta)(s+\norm{\varphi}_\infty)N} \cdot  \left(2^\delta\varepsilon_0^{\delta(1-\delta)}\right)^N \cdot
    \tau_n^{sN-\sup_{E_n}S_N\varphi}. \]
Thus 
\begin{equation*}
  \begin{split}
  & \mathcal{H}^{sN +2\delta(sN - m_N)+\delta N}_{2\varepsilon^{1-\delta}} (\mathcal{X}, \mathbf{d}_N, S_N \varphi) \\
  & \leq  \sum_{n=1}^\infty \left\{2^{1+(1+2\delta)(s+\norm{\varphi}_\infty)} \cdot M(\tau_n)^{1/L_n} \cdot 
     \left(2^\delta\varepsilon_0^{\delta(1-\delta)}\right) \right\}^N 
           \cdot  \tau_n^{sN-\sup_{E_n} S_N\varphi}  \\
    &  \leq   \sum_{n=1}^\infty \left\{2^{2+\delta+(1+2\delta)(s+\norm{\varphi}_\infty)}\cdot \varepsilon_0^{\delta(1-\delta)}\right\}^N 
               \cdot  \tau_n^{sN-\sup_{E_n} S_N\varphi}  
      \quad   \text{by $1/L_n = \tau_n^\delta$ and (\ref{eq: choice of varepsilon_0 1})} \\
         & \leq \sum_{n=1}^\infty   \tau_n^{sN-\sup_{E_n} S_N\varphi}    
         \quad \text{by (\ref{eq: choice of varepsilon_0 2})} \\
     & < 1   \quad \text{by (\ref{eq: choice of tau_n in the proof of L^1 mean hausdorff dimension is equal to mean hausdorff dimension})}.  
  \end{split}   
\end{equation*}
Therefore 
\begin{equation*}
   \begin{split}
   \dim_{\mathrm{H}}(\mathcal{X},\mathbf{d}_N,S_N\varphi,2\varepsilon^{1-\delta}) & \leq 
   sN+ 2\delta(sN-m_N) + \delta N  \\
   & \leq s N + 2\delta (sN + \norm{\varphi}_\infty N) + \delta N.
   \end{split}
\end{equation*}   
Divide this by $N$. 
Let $N\to \infty $ and $\varepsilon \to 0$:
\[  \mdim_{\mathrm{H}}(\mathcal{X},T,\mathbf{d},\varphi) \leq 
     s + 2\delta(s+\norm{\varphi}_\infty) + \delta. \]   
Let $\delta\to 0$ and $s\to \mdim_{\mathrm{H},L^1}(\mathcal{X},T,\mathbf{d},\varphi)$:
\[  \mdim_{\mathrm{H}}(\mathcal{X},T,\mathbf{d},\varphi)  \leq
    \mdim_{\mathrm{H},L^1}(\mathcal{X},T,\mathbf{d},\varphi). \]
\end{proof}

\begin{remark}
The same argument also proves that $\underline{\mdim}_{\mathrm{H}}(\mathcal{X},T,\mathbf{d},\varphi)$ is equal to 
\[  \lim_{\varepsilon\to 0}\left(\liminf_{N\to \infty}
     \frac{\dim_{\mathrm{H}}(\mathcal{X},\overline{\mathbf{d}}_N, S_N\varphi, \varepsilon)}{N}\right) \]
     if $(\mathcal{X},\mathbf{d})$ has the tame growth of covering numbers.
\end{remark}

\subsection{Proof of Theorem \ref{theorem: dynamical Frostman's lemma}}
\label{subsection: proof of dynamical Frostman's lemma}

We need the next two lemmas for proving Theorem \ref{theorem: dynamical Frostman's lemma}.
Let $(\mathcal{X},\mathbf{d})$ be a compact metric space.
For $\varepsilon>0$ and $s\geq 0$ we set 
$\mathcal{H}^s_\varepsilon(\mathcal{X},\mathbf{d}) = \mathcal{H}^s_\varepsilon(\mathcal{X},\mathbf{d}, 0)$.
Namely 
 \[ \mathcal{H}^s_\varepsilon(\mathcal{X},\mathbf{d}) = 
    \inf \left\{\sum_{i=1}^\infty \left(\diam E_i\right)^{s} \middle|
     \mathcal{X} = \bigcup_{i=1}^\infty E_i \text{ with } \diam E_i < \varepsilon 
                        \text{ for all $i\geq 1$}\right\}. \]
We define $\dim_{\mathrm{H}}(\mathcal{X},\mathbf{d},\varepsilon)$ as the supremum of $s\geq 0$ satisfying 
$\mathcal{H}^s_\varepsilon(\mathcal{X},\mathbf{d}) \geq 1$.

\begin{lemma} \label{lemma: Howroyd Frostman's lemma}
Let $0<c<1$. There exists $0<\delta_0(c)<1$ depending only on $c$ and satisfying the following statement.
For any compact metric space $(\mathcal{X},\mathbf{d})$ and $0<\delta\leq \delta_0(c)$
there exists a Borel probability measure $\nu$ on $\mathcal{X}$ such that 
\[  \nu(E) \leq \left(\diam E\right)^{c \cdot \dim_{\mathrm{H}}(\mathcal{X},\mathbf{d},\delta)} \quad 
    \text{for all $E\subset \mathcal{X}$ with $\diam E < \frac{\delta}{6}$}. \]
\end{lemma}

\begin{proof}
This follows from Howroyd's approach \cite{Howroyd} to Frostman's lemma
for general compact metric spaces.
See \cite[Corollary 4.4]{Lindenstrauss--Tsukamoto double VP} for the details.                                                                                                                                                                                                                                                                                                        
\end{proof}

\begin{lemma}  \label{lemma: optimal transport}
Let $A$ be a finite set.
Suppose that probability measures $\mu_n$ on $A$ converge to some $\mu$ in the weak$^*$ topology.
Then there exist probability measures $\pi_n$ $(n\geq 1)$ on $A\times A$ such that 
\begin{itemize}
 \item $\pi_n$ is a coupling between $\mu_n$ and $\mu$. Namely the first and second marginals of $\pi_n$ are
          given by $\mu_n$ and $\mu$ respectively.
 \item  $\pi_n$ converge to $(\mathrm{id}\times \mathrm{id})_* \mu$ in the weak$^*$ topology. Namely 
         \[ \pi_n(a,b) \to \begin{cases}
                               0   &(a\neq b) \\  
                               \mu(a)  &(a=b)
                               \end{cases}. \] 
\end{itemize}
\end{lemma}

\begin{proof}
This follows from a general fact on optimal transport that the Wasserstein distance metrizes the weak$^*$ topology
\cite[Theorem 6.9]{Villani}.
See \cite[Appendix]{Lindenstrauss--Tsukamoto rate distortion} for an elementary proof.
\end{proof}

Theorem \ref{theorem: dynamical Frostman's lemma} is contained in the following statement.
Recall that we have denoted by $\mathscr{M}^T(\mathcal{X})$ the set of invariant probability measures on a dynamical system 
$(\mathcal{X},T)$.

\begin{theorem}[$\supset$ Theorem \ref{theorem: dynamical Frostman's lemma}] \label{theorem: dynamical Frostman's lemma refined}
Let $(\mathcal{X},T)$ be a dynamical system with a metric $\mathbf{d}$ and
a continuous function $\varphi:\mathcal{X}\to \mathbb{R}$.
Then 
\[  \mdim_{\mathrm{H}, L^1}(\mathcal{X},T,\mathbf{d},\varphi) 
    \leq \sup_{\mu\in \mathscr{M}^T(\mathcal{X})} 
    \left(\underline{\rdim}(\mathcal{X},T,\mathbf{d},\mu) + \int_{\mathcal{X}} \varphi d\mu\right). \]
In particular (by Lemma \ref{lemma: L^1 mean hausdorff dimension is equal to mean hausdorff dimension}) 
if $(\mathcal{X},\mathbf{d})$ has the tame growth of covering numbers then 
\[  \mdim_{\mathrm{H}}(\mathcal{X}, T, \mathbf{d},\varphi) \leq 
     \sup_{\mu\in \mathscr{M}^T(\mathcal{X})} 
    \left(\underline{\rdim}(\mathcal{X},T,\mathbf{d},\mu) + \int_{\mathcal{X}} \varphi d\mu\right). \]
\end{theorem}

\begin{proof}
We extend the definition of $\overline{\mathbf{d}}_n$.
For $x=(x_0,x_1,\dots,x_{n-1})$ and $y=(y_0,y_1,\dots,y_{n-1})$ in $\mathcal{X}^n$, we set 
\[  \overline{\mathbf{d}}_n(x,y) = \frac{1}{n}\sum_{i=0}^{n-1} \mathbf{d}(x_i,y_i). \]

Let $0<c<1$ and $s<\mdim_{\mathrm{H},L^1}(\mathcal{X},T,\mathbf{d},\varphi)$ be arbitrary.
We will construct an invariant probability measure $\mu$ on $\mathcal{X}$ satisfying 
\begin{equation} \label{eq: invariant measure capturing dynamical complexity}
   \underline{\rdim}(\mathcal{X},T,\mathbf{d},\mu) + \int_{\mathcal{X}} \varphi \, d\mu 
   \geq c s - (1-c) \norm{\varphi}_\infty.
\end{equation}
Letting $c\to 1$ and $s\to \mdim_{\mathrm{H},L^1}(\mathcal{X},T,\mathbf{d},\varphi)$, we get the statement of the theorem.

Take $\eta>0$ satisfying $\mdim_{\mathrm{H},L^1}(\mathcal{X},T,\mathbf{d},\varphi) -2\eta> s$.
Let $\delta_0 = \delta_0(c)\in (0,1)$ be a constant given by Lemma \ref{lemma: Howroyd Frostman's lemma}.
There exist $0<\delta<\delta_0$ and a sequence $n_1<n_2<n_3<\dots \to \infty$ satisfying 
\[  \dim_{\mathrm{H}}(\mathcal{X},\overline{\mathbf{d}}_{n_k},S_{n_k}\varphi,\delta) > (s+2\eta)n_k. \]

\begin{claim}
There exists $t\in [-\norm{\varphi}_\infty, \norm{\varphi}_\infty]$ such that for infinitely many $n_k$
\[   \dim_{\mathrm{H}}\left(\left(\frac{S_{n_k}\varphi}{n_k}\right)^{-1}[t,t+\eta], \overline{\mathbf{d}}_{n_k}, \delta\right) 
      \geq (s-t)n_k. \]
\end{claim}

\begin{proof}
We have $\mathcal{H}_\delta^{(s+2\eta)n_k}(\mathcal{X}, \overline{\mathbf{d}}_{n_k},S_{n_k}\varphi) \geq 1$.
Set $m = \lceil 2\norm{\varphi}_\infty/\eta \rceil$ and consider 
\[  \mathcal{X} = \bigcup_{l=0}^{m-1} \left(\frac{S_{n_k}\varphi}{n_k}\right)^{-1}
     \left[-\norm{\varphi}_\infty + l\eta, -\norm{\varphi}_\infty + (l+1)\eta\right]. \]
Then there exists $t\in \{-\norm{\varphi}_\infty + l\eta |\, l=0,1,\dots, m-1\}$ such that for infinitely many 
$n_k$
\[  \mathcal{H}^{(s+2\eta)n_k}_\delta\left(\left(\frac{S_{n_k}\varphi}{n_k}\right)^{-1}[t,t+\eta], \overline{\mathbf{d}}_{n_k}, S_{n_k}\varphi\right)
     \geq \frac{1}{m}. \]
Since $(s+2\eta)n_k - S_{n_k}\varphi \geq (s+2\eta)n_k - (t+\eta)n_k = (s-t)n_k + \eta n_k$ 
on the set $(S_{n_k}\varphi/n_k)^{-1}[t,t+\eta]$, 
\begin{equation*}
   \begin{split}
    \mathcal{H}^{(s+2\eta)n_k}_\delta\left(\left(\frac{S_{n_k}\varphi}{n_k}\right)^{-1}[t,t+\eta], \overline{\mathbf{d}}_{n_k}, S_{n_k}\varphi\right)
    &\leq  \mathcal{H}^{(s-t)n_k+\eta n_k}_\delta
    \left(\left(\frac{S_{n_k}\varphi}{n_k}\right)^{-1}[t,t+\eta], \overline{\mathbf{d}}_{n_k}\right)  \\
    &\leq \delta^{\eta n_k}     \cdot
     \mathcal{H}^{(s-t)n_k}_\delta \left(\left(\frac{S_{n_k}\varphi}{n_k}\right)^{-1}[t,t+\eta], \overline{\mathbf{d}}_{n_k}\right).
   \end{split}
\end{equation*}   
Hence for infinitely many $n_k$
\[   \mathcal{H}^{(s-t)n_k}_\delta \left(\left(\frac{S_{n_k}\varphi}{n_k}\right)^{-1}[t,t+\eta], \overline{\mathbf{d}}_{n_k}\right)
      \geq \frac{\delta^{-\eta n_k}}{m}. \]
The right-hand side is larger than one for sufficiently large $n_k$. Then for such $n_k$
\[   \dim_{\mathrm{H}}\left(\left(\frac{S_{n_k}\varphi}{n_k}\right)^{-1}[t,t+\eta], \overline{\mathbf{d}}_{n_k}, \delta\right) 
      \geq (s-t)n_k. \]
\end{proof}

By choosing a subsequence of $\{n_k\}$ (also denoted by $\{n_k\}$), we assume that the condition
\[   \dim_{\mathrm{H}}\left(\left(\frac{S_{n_k}\varphi}{n_k}\right)^{-1}[t,t+\eta], \overline{\mathbf{d}}_{n_k}, \delta\right) 
      \geq (s-t)n_k. \]
holds for all $n_k$.
Noting $0<\delta<\delta_0(c)$, we apply Lemma \ref{lemma: Howroyd Frostman's lemma} to 
the subspace $\left(\frac{S_{n_k}\varphi}{n_k}\right)^{-1}[t,t+\eta] \subset \mathcal{X}$.
Then we can find a Borel probability measure $\nu_k$ supported on $\left(\frac{S_{n_k}\varphi}{n_k}\right)^{-1}[t,t+\eta]$ such that 
\begin{equation}  \label{eq: scaling law of measure nu_k}
    \nu_k(E) \leq \left(\diam(E, \overline{\mathbf{d}}_{n_k}) \right)^{c(s-t)n_k} \quad 
    \text{for all $E\subset \mathcal{X}$ with $\diam (E, \overline{\mathbf{d}}_{n_k}) < \frac{\delta}{6}$}.
\end{equation}
Notice that $\nu_k$ is not necessarily invariant under $T$.
Set 
\[  \mu_k = \frac{1}{n_k} \sum_{n=0}^{n_k-1} T^n_* \nu_k. \]
By choosing a subsequence (also denoted by $\{n_k\}$ again) we can assume that 
$\mu_k$ converges to some $\mu\in \mathscr{M}^T(\mathcal{X})$ in the weak$^*$ topology.
Then 
\[  \int_{\mathcal{X}}\varphi \, d\mu_k \to \int_{\mathcal{X}} \varphi \, d\mu \quad (k\to \infty).  \]
On the other hand
\begin{equation*}
    \int_{\mathcal{X}}\varphi \, d\mu_k  = \frac{1}{n_k} \sum_{n=0}^{n_k-1} \int_{\mathcal{X}} \varphi \circ T^n \, d\nu_k
    = \int_{\mathcal{X}} \frac{S_{n_k} \varphi}{n_k} \, d\nu_k \geq t 
\end{equation*}
since $\nu_k$ is supported on the set $(S_{n_k}\varphi/n_k)^{-1}[t,t+\eta]$.
Hence 
\[   \int_{\mathcal{X}} \varphi \, d\mu \geq t. \]

We will prove 
\begin{equation} \label{eq: rate distortion dimension of mu}
     \underline{\rdim}(\mathcal{X},T,\mathbf{d},\mu) \geq c(s-t). 
\end{equation}     
Assuming this for the moment, we get (\ref{eq: invariant measure capturing dynamical complexity}) 
(recall $|t|\leq \norm{\varphi}_\infty$): 
\[  \underline{\rdim}(\mathcal{X},T,\mathbf{d},\mu) + \int_{\mathcal{X}}\varphi \, d\mu 
    \geq c(s-t) + t = c s + (1-c) t \geq cs -(1-c)\norm{\varphi}_\infty. \]
So the rest of the problem is to prove (\ref{eq: rate distortion dimension of mu}).
This part of the proof is the same as \cite[Section 4.3]{Lindenstrauss--Tsukamoto double VP}.
The method is a ``rate distortion theory version'' of Misiurewicz's technique \cite{Misiurewicz}
(a famous proof of the standard variational principle) first developed in \cite{Lindenstrauss--Tsukamoto rate distortion}.
The paper \cite[Section 4.3]{Lindenstrauss--Tsukamoto double VP} explained more background ideas behind the proof, which 
we do not repeat here.

Let $\varepsilon$ be an arbitrary positive number with $2\varepsilon \log(1/\varepsilon) \leq \delta/10$.
We will show a lower bound on the rate distortion function of the form 
\[  R(\mathbf{d},\mu,\varepsilon) \geq c(s-t) \log(1/\varepsilon) + \text{small error terms}. \]
Let $X$ and $Y=(Y_0,Y_1,\dots, Y_{m-1})$ be random variables defined on a probability space $(\Omega, \mathbb{P})$
such that $X, Y_0,\dots, Y_{m-1}$ take values in $\mathcal{X}$ and satisfy 
\[  \mathrm{Law} X = \mu, \quad \mathbb{E}\left(\frac{1}{m} \sum_{j=0}^{m-1} \mathbf{d}(T^j X, Y_j) \right) < \varepsilon. \]
We would like to establish a lower bound on the mutual information $I(X;Y)$.
For this purpose (see Remark \ref{remark: rate distortion function}), 
we can assume that $Y$ takes only finitely many values.
Let $\mathcal{Y}\subset \mathcal{X}^m$ be the (finite) set of possible values of $Y$. 

We choose $\tau>0$ satisfying 
\begin{equation} \label{eq: choice of tau in the proof of dynamical Frostman's lemma}
  \tau \leq \min\left(\frac{\varepsilon}{3},\frac{\delta}{20}\right), \quad 
  \frac{\tau}{2} + \mathbb{E}\left(\frac{1}{m} \sum_{j=0}^{m-1} \mathbf{d}(T^j X, Y_j)\right) < \varepsilon. 
\end{equation}
We take a measurable partition $\mathcal{P}= \{P_1,\dots, P_L\}$ of $\mathcal{X}$ such that 
for all $1\leq l\leq L$
\[  \diam (P_l, \mathbf{d}) < \frac{\tau}{2}, \quad 
    \mu(\partial P_l) = 0. \]
Pick $p_l\in P_l$ and set $A=\{p_1,\dots, p_L\}$.
We define $\mathcal{P}:\mathcal{X}\to A$ by $\mathcal{P}(P_l) = \{p_l\}$.
For $n\geq 1$ we define $\mathcal{P}^n:\mathcal{X}\to A^n$ by 
$\mathcal{P}^n(x) = \left(\mathcal{P}(x), \mathcal{P}(T x), \dots, \mathcal{P}(T^{n-1}x)\right)$.

\begin{claim} \label{claim: pushforward measure has the scaling law}
The pushforward measure $\mathcal{P}^{n_k}_*\nu_k$ satisfies 
\[   \mathcal{P}^{n_k}_*\nu_k(E) \leq \left(\tau + \diam(E, \overline{\mathbf{d}}_{n_k})\right)^{c(s-t)n_k} \quad 
      \text{for all $E\subset A^{n_k}$ with $\diam (E,\overline{\mathbf{d}}_{n_k}) < \frac{\delta}{10}$}. \]
\end{claim}

\begin{proof}
From $\diam (P_l, \mathbf{d}) < \tau/2$ and $\tau \leq \delta/20$, if $\diam (E,\overline{\mathbf{d}}_{n_k}) < \delta/10$ then
\[ \diam \left(\left(\mathcal{P}^{n_k}\right)^{-1}E, \overline{\mathbf{d}}_{n_k}\right) 
    < \tau + \diam (E, \overline{\mathbf{d}}_{n_k}) < \frac{\delta}{6}. \]
By (\ref{eq: scaling law of measure nu_k}), the measure 
$\mathcal{P}^{n_k}_*\nu_k(E) = \nu_k\left(\left(\mathcal{P}^{n_k}\right)^{-1}E\right)$ is bounded by 
\[    \left(\diam \left(\left(\mathcal{P}^{n_k}\right)^{-1}E, \overline{\mathbf{d}}_{n_k}\right)\right)^{c(s-t)n_k}
       < \left(\tau + \diam (E, \overline{\mathbf{d}}_{n_k})\right)^{c(s-t)n_k}. \]
\end{proof}

From $\mu_k\to \mu$ and $\mu(\partial P_l) =0$, we have $\mathcal{P}^{m}_*\mu_k\to \mathcal{P}^m_*\mu$.
By Lemma \ref{lemma: optimal transport}, there exists a coupling $\pi_k$ between $\mathcal{P}^m_*\mu_k$ and 
$\mathcal{P}^m_*\mu$ such that $\pi_k\to (\mathrm{id}\times \mathrm{id})_*\mathcal{P}^m_*\mu$.
Let $X(k)$ be a random variable coupled to $\mathcal{P}^m(X)$ such that it takes values in $A^m$ and
$\mathrm{Law}\left(X(k),\mathcal{P}^m(X)\right) = \pi_k$. In particular, $\mathrm{Law} X(k) = \mathcal{P}^m_*\mu_k$.
From $\pi_k\to (\mathrm{id}\times \mathrm{id})_*\mathcal{P}^m_*\mu$,
\[  \mathbb{E} \overline{\mathbf{d}}_m\left(X(k),\mathcal{P}^m(X)\right) \to 0. \]
The random variables $X(k)$ and $Y$ are coupled by the probability mass function 
\[  \sum_{x'\in A^m} \pi_k(x,x') \mathbb{P}(Y=y| \mathcal{P}^m(X)=x') \quad (x\in A^m, y\in \mathcal{Y}), \]
which converges to $\mathbb{P}(\mathcal{P}^m(X)=x, Y=y)$. Then by Lemma \ref{lemma: convergence of mutual information}
\begin{equation} \label{eq: convergence of mutual information between X(k) and Y}
     I\left(X(k);Y\right) \to I\left(\mathcal{P}^m(X);Y\right).
\end{equation}
By the triangle inequality 
\begin{equation*}
   \begin{split}
   \overline{\mathbf{d}}_m\left(X(k),Y\right)  \leq 
   & \overline{\mathbf{d}}_m\left(X(k), \mathcal{P}^m(X)\right) + \overline{\mathbf{d}}_m\left(\mathcal{P}^m(X), (X,TX,\dots, T^{m-1}X)\right)\\
   & + \overline{\mathbf{d}}_m\left((X,TX,\dots, T^{m-1}X), Y\right).
   \end{split}
\end{equation*}
We have $\mathbb{E}\overline{\mathbf{d}}_m\left(X(k),\mathcal{P}^m(X)\right) \to 0$, $\diam (P_l, \mathbf{d}) < \tau/2$ for all $1\leq l\leq L$
and $(\tau/2) + \mathbb{E} \overline{\mathbf{d}}_m\left((X,TX,\dots, T^{m-1}X), Y\right) < \varepsilon$ 
in (\ref{eq: choice of tau in the proof of dynamical Frostman's lemma}).
Then
\begin{equation}  \label{eq: expected distance between X(k) and Y is small}
  \mathbb{E} \overline{\mathbf{d}}_m\left(X(k),Y\right) < \varepsilon \quad \text{for sufficiently large $k$}.
\end{equation}

For $x=(x_0,\dots,x_{n-1})\in \mathcal{X}^n$ and $0\leq a\leq b < n$ we write
$x_a^b = (x_a, x_{a+1},\dots,x_b)$.
We consider a conditional probability mass function 
\[  \rho_k(y|x) = \mathbb{P}(Y=y|\, X(k)=x) \]
for $x,y\in \mathcal{X}^m$ with $\mathbb{P}(X(k)=x) = \mathcal{P}^m_*\mu_k(x) >0$.
Fix a point $a\in \mathcal{X}$. We denote by $\delta_a(\cdot)$ the delta probability measure at $a$ on $\mathcal{X}$.
Let $n_k= mq+r$ with $m \leq r < 2m$.
Let $x,y\in \mathcal{X}^{n_k}$ with $\mathcal{P}^{n_k}_*\nu_k(x)>0$.
For $0\leq j < m$ we define a conditional probability mass function 
\begin{equation} \label{eq: definition of conditional probability measure sigma_{k,j}}
     \sigma_{k,j}(y|x) = \prod_{i=0}^{q-1} \rho_k\left(y_{j+i m}^{j+i m+m-1}| x_{j+i m}^{j+i m+m-1}\right) \cdot 
                            \prod_{n\in [0,j)\cup [mq+j, n_k)} \delta_a(y_n). 
\end{equation}                            
Set 
\begin{equation} \label{eq: definition of conditional probability measure sigma_k}
    \sigma_k(y|x) = \frac{\sigma_{k,0}(y|x)+\sigma_{k,1}(y|x)+\dots+\sigma_{k,m-1}(y|x)}{m}. 
\end{equation}    

Let $X'(k)$ be a random variable taking values in $\mathcal{X}$ with $\mathrm{Law} X'(k) = \nu_k$.
Set $Z(k) = \mathcal{P}^{n_k}(X'(k))$.
We define a random variable $W(k)$ taking values in $\mathcal{X}^{n_k}$ and coupled to $Z(k)$ by the condition 
\[  \mathbb{P}\left(W(k) = y\middle| Z(k) = x\right) = \sigma_k(y|x). \] 
For $0\leq j <m$ we also define $W(k,j)$ by 
\[  \mathbb{P}\left(W(k,j) = y\middle| Z(k) = x\right) = \sigma_{k,j}(y|x). \]

\begin{claim}  \label{claim: distance between Z(k) and W(k)}
\[  \mathbb{E}\overline{\mathbf{d}}_{n_k}(Z(k), W(k)) < \varepsilon \quad 
     \text{for sufficiently large $k$}. \]
\end{claim}

\begin{proof}
From the definition of $\sigma_k$ in (\ref{eq: definition of conditional probability measure sigma_k})
\begin{equation} \label{eq: distance between Z(k) and W(k)}
    \mathbb{E}\overline{\mathbf{d}}_{n_k}(Z(k), W(k))  = \frac{1}{m} \sum_{j=0}^{m-1} 
     \mathbb{E} \overline{\mathbf{d}}_{n_k} (Z(k), W(k,j)). 
\end{equation}     
From $Z(k) = \mathcal{P}^{n_k}(X'(k))$, the distance $\overline{\mathbf{d}}_{n_k} (Z(k), W(k,j))$ is bounded by 
\[  \frac{r \cdot \diam (\mathcal{X},\mathbf{d})}{n_k} + \frac{m}{n_k}
     \sum_{i=0}^{q-1} \overline{\mathbf{d}}_m\left(\mathcal{P}^m(T^{j+im} X'(k)), W(k,j)_{j+im}^{j+im+m-1}\right). \]
From $\mathrm{Law} X'(k) = \nu_k$ and the definition of $\sigma_{k,j}$ in (\ref{eq: definition of conditional probability measure sigma_{k,j}}), 
\[  \mathbb{E} \overline{\mathbf{d}}_m\left(\mathcal{P}^m(T^{j+im} X'(k)), W(k,j)_{j+im}^{j+im+m-1}\right)
     = \sum_{x,y\in \mathcal{X}^m} \overline{\mathbf{d}}_m (x,y) \rho_k(y|x) \mathcal{P}_*^m T^{j+im}_*\nu_k(x), \]     
where the right-hand side is a finite sum because $\rho_k(y|x) \mathcal{P}_*^m T^{j+im}_*\nu_k(x)$
can be nonzero only for $x\in A^m$ and $y\in \mathcal{Y}$.
Hence (\ref{eq: distance between Z(k) and W(k)}) is bounded by 
\[  \frac{r\cdot \diam(\mathcal{X},\mathbf{d})}{n_k}  + 
    \underbrace{\sum_{x,y\in \mathcal{X}^m} \overline{\mathbf{d}}_m (x,y) \rho_k(y|x) 
     \left(\frac{1}{n_k}\sum_{\substack{0\leq i <q \\ 0\leq j < m}} \mathcal{P}^m_*T^{j+im}_*\nu_k(x)  \right)}_{(I)}. \]
The term $(I)$ is estimated by 
\begin{equation*}
   \begin{split}
   (I)  &\leq \sum_{x,y\in \mathcal{X}^m} \overline{\mathbf{d}}_m (x,y) \rho_k(y|x) 
     \left(\frac{1}{n_k}\sum_{n=0}^{n_k-1} \mathcal{P}^m_*T^{n}_*\nu_k(x)  \right) \\
   & = \sum_{x,y\in \mathcal{X}^m} \overline{\mathbf{d}}_m(x,y) \rho_k(y|x) \mathcal{P}^m_*\mu_k(x)
     \quad \text{by $\mu_k = \frac{1}{n_k} \sum_{n=0}^{n_k-1} T^n_*\nu_k$} \\
     & = \mathbb{E}\overline{\mathbf{d}}_m(X(k),Y).
   \end{split}
\end{equation*}   
Therefore 
\[   \mathbb{E}\overline{\mathbf{d}}_{n_k}(Z(k), W(k)) \leq  \frac{r\cdot \diam(\mathcal{X},\mathbf{d})}{n_k} +
      \mathbb{E}\overline{\mathbf{d}}_m(X(k),Y). \]
Recall $r\leq 2m$. The term $\mathbb{E}\overline{\mathbf{d}}_m(X(k),Y)$ is smaller than $\varepsilon$ for large $k$ by 
(\ref{eq: expected distance between X(k) and Y is small}).
Thus $\mathbb{E}\overline{\mathbf{d}}_{n_k}(Z(k), W(k)) < \varepsilon$ for large $k$.
\end{proof}

\begin{claim} \label{claim: mutual information between Z(k) and W(k)}
\[   \frac{1}{n_k} I(Z(k); W(k)) \leq \frac{1}{m} I(X(k);Y). \]
\end{claim}

\begin{proof}
The mutual information is a convex function of conditional probability measure (Lemma \ref{lemma: convexity of mutual information}).
Hence
\[  I(Z(k); W(k)) \leq \frac{1}{m} \sum_{j=0}^{m-1}I(Z(k); W(k,j)). \]
By the subadditivity under conditional independence 
(Lemma \ref{lemma: subadditivity of mutual information}),
\[  I(Z(k);W(k,j)) \leq \sum_{i=0}^{q-1} I(Z(k); W(k,j)_{j+im}^{j+im+m-1}). \]
The term $I(Z(k); W(k,j)_{j+im}^{j+im+m-1})$ is equal to 
\[ I\left(\mathcal{P}^m\left(T^{j+im}X'(k)\right); W(k,j)_{j+im}^{j+im+m-1}\right)
    = I(\mathcal{P}^m_*T^{j+im}_* \nu_k, \rho_k). \]
Therefore
\begin{equation*}
   \begin{split}
   \frac{m}{n_k} I\left(Z(k);W(k)\right)  \leq &
     \frac{1}{n_k} \sum_{\substack{0\leq j < m\\ 0\leq i<q}} I\left(\mathcal{P}^m_*T^{j+im}_*\nu_k, \rho_k\right) \\
     \leq &  \frac{1}{n_k} \sum_{n=0}^{n_k-1} I\left(\mathcal{P}^m_*T^n_*\nu_k, \rho_k\right) \\
     \leq &  I\left(\frac{1}{n_k}\sum_{n=0}^{n_k-1}\mathcal{P}^m_*T^n_*\nu_k, \rho_k\right) 
       \text{ by the concavity in Lemma \ref{lemma: convexity of mutual information}} \\
      =  & I\left(\mathcal{P}^m_*\mu_k, \rho_k\right)  \quad 
       \text{by $\mu_k = \frac{1}{n_k} \sum_{n=0}^{n_k-1} T^n_*\nu_k$} \\
      = &  I\left(X(k); Y\right).
   \end{split}
\end{equation*}   
\end{proof}

Recall $2\varepsilon \log(1/\varepsilon) \leq \delta/10$ and $\tau\leq \min(\varepsilon/3,\delta/20)$.
The measure $\mathrm{Law} Z(k) = \mathcal{P}^{n_k}_*\nu_k$ satisfies the ``scaling law'' given by 
Claim \ref{claim: pushforward measure has the scaling law}. 
Then we apply Lemma \ref{lemma: geometry and mutual information} to $(Z(k), W(k))$
with Claim \ref{claim: distance between Z(k) and W(k)} 
($\mathbb{E}\overline{\mathbf{d}}_{n_k}(Z(k),W(k)) < \varepsilon$ for $k\gg 1$), which provides
\[  I(Z(k); W(k)) \geq c(s-t) n_k \log(1/\varepsilon) - K \left(c(s-t) n_k + 1\right) \quad 
     \text{for large $k$}. \]
Here $K$ is a universal positive constant.
From Claim \ref{claim: mutual information between Z(k) and W(k)}, for large $k$
\[  \frac{1}{m} I(X(k);Y) \geq \frac{1}{n_k} I(Z(k);W(k)) \geq  c(s-t)  \log(1/\varepsilon) - K\left(c(s-t)  + \frac{1}{n_k}\right). \]
Since $I(X(k); Y) \to I\left(\mathcal{P}^m(X); Y\right)$ by (\ref{eq: convergence of mutual information between X(k) and Y}), 
we get 
\[  \frac{1}{m}  I\left(\mathcal{P}^m(X); Y\right) \geq c(s-t) \log (1/\varepsilon) - cK (s-t). \]
By the data-processing inequality (Lemma \ref{lemma: data-processing inequality})
\[ \frac{1}{m} I(X;Y) \geq \frac{1}{m}  I\left(\mathcal{P}^m(X); Y\right) \geq c(s-t) \log (1/\varepsilon) - cK (s-t). \]
This proves that for any $\varepsilon >0$ with $2\varepsilon \log(1/\varepsilon) \leq \delta/10$
\[  R(\mathbf{d},\mu,\varepsilon) \geq  c(s-t) \log (1/\varepsilon) - cK (s-t). \]
Thus we get (\ref{eq: rate distortion dimension of mu}):
\[  \underline{\rdim}(\mathcal{X},T,\mathbf{d},\mu)
    = \liminf_{\varepsilon\to 0} \frac{R(\mathbf{d},\mu,\varepsilon)}{\log(1/\varepsilon)}  \geq c (s-t). \]
This establishes the proof of the theorem.
\end{proof}

\subsection{Proof of Corollary \ref{corollary: mean dimension, rate distortion dimension and metric mean dimension}}
\label{subsection: proof of Corollary of dynamical Frostman's lemma}

\begin{corollary}[$=$ Corollary \ref{corollary: mean dimension, rate distortion dimension and metric mean dimension}]
 \label{corollary: mean dimension, rate distortion dimension and metric mean dimension restated}
 Let $(\mathcal{X},T)$ be a dynamical system with a metric $\mathbf{d}$ and a continuous function 
 $\varphi:\mathcal{X}\to \mathbb{R}$. Then
\begin{equation*}
  \begin{split}
    \mdim(\mathcal{X}, T, \varphi) & \leq  \sup_{\mu\in \mathscr{M}^T(\mathcal{X})}
    \left(\underline{\rdim}(\mathcal{X},T, \mathbf{d}, \mu) + \int_{\mathcal{X}} \varphi \, d\mu\right)  \\
   &  \leq  \sup_{\mu\in \mathscr{M}^T(\mathcal{X})} \left(\overline{\rdim}(\mathcal{X},T,\mathbf{d}, \mu) +
    \int_{\mathcal{X}} \varphi \, d\mu\right)
   \leq \overline{\mdim}_{\mathrm{M}}(\mathcal{X}, T, \mathbf{d}, \varphi).   
  \end{split} 
\end{equation*}
\end{corollary}

\begin{proof}
From Proposition \ref{prop: metric mean dimension bounds rate distortion dimension restated},
\[   \sup_{\mu\in \mathscr{M}^T(\mathcal{X})} \left(\overline{\rdim}(\mathcal{X},T,\mathbf{d}, \mu) +
    \int_{\mathcal{X}} \varphi \, d\mu\right)
   \leq \overline{\mdim}_{\mathrm{M}}(\mathcal{X}, T, \mathbf{d}, \varphi).   \]
From Lemma \ref{lemma: tame growth is mild}, we can find a metric $\mathbf{d}'$ on $\mathcal{X}$ such that 
$\mathbf{d}'(x,y) \leq \mathbf{d}(x,y)$ and that $(\mathcal{X},\mathbf{d}')$ has the tame growth of covering numbers.
Then 
\begin{equation*}
  \begin{split}
    \mdim(\mathcal{X}, T, \varphi) & \leq    \mdim_{\mathrm{H}}(\mathcal{X},T,\mathbf{d}',\varphi)   
    \quad  \text{by Theorem \ref{theorem: mean Hausdorff dimension bounds mean dimension restated}} \\
    & \leq    \sup_{\mu\in \mathscr{M}^T(\mathcal{X})}
    \left(\underline{\rdim}(\mathcal{X},T, \mathbf{d}', \mu) + \int_{\mathcal{X}} \varphi \, d\mu\right)  
   \quad \text{by Theorem \ref{theorem: dynamical Frostman's lemma refined}}    \\
   &  \leq  \sup_{\mu\in \mathscr{M}^T(\mathcal{X})} \left(\underline{\rdim}(\mathcal{X},T,\mathbf{d}, \mu) +
    \int_{\mathcal{X}} \varphi \, d\mu\right) \quad 
    \text{by $\mathbf{d}'(x,y) \leq \mathbf{d}(x,y)$}.
  \end{split} 
\end{equation*}
\end{proof}

\section{Dynamical Pontrjagin--Schnirelmann's theorem: proof of Theorem \ref{theorem: dynamical PS theorem}}
\label{section: dynamical PS theorem}

We prove Theorem \ref{theorem: dynamical PS theorem} here.
The proof is given in \S \ref{subsection: proof of dynamical PS theorem}.
The first two subsections are preparations.
This section is rather technically hard.
The paper \cite[Section 5.1]{Lindenstrauss--Tsukamoto double VP} explained more backgrounds.

\subsection{Preparations on combinatorial topology}  \label{subsection: preparations on combinatorial topology}

In this subsection we prepare some definitions and results about simplicial complex.
Recall that we have assumed that simplicial complexes are always finite (having only finitely many vertices).
Let $P$ be a simplicial complex.
We denote by $\ver(P)$ the set of vertices of $P$.
For a vertex $v$ of $P$ we define the \textbf{open star} $O_P(v)$ as the union of open simplexes of $P$ one of whose 
vertex is $v$. 
Here $\{v\}$ itself is an open simplex. So $O_P(v)$ is an open neighborhood of $v$, and $\{O_P(v)\}_{v\in \ver(P)}$
forms an open cover of $P$.
For a simplex $\Delta\subset P$ we set $O_P(\Delta) = \bigcup_{v\in \ver(\Delta)} O_P(v)$.

Let $P$ and $Q$ be simplicial complexes.
A map $f:P\to Q$ is said to be \textbf{simplicial} if for every simplex $\Delta\subset P$
the image $f(\Delta)$ is a simplex in $Q$ and 
\[ f\left(\sum_{v\in \ver(\Delta)} \lambda_v v \right)  = \sum_{v\in \ver(\Delta)} \lambda_v f(v), \]
where $0\leq \lambda_v \leq 1$ and $\sum_{v\in \ver(\Delta)} \lambda_v = 1$.

Let $V$ be a real vector space.
A map $f:P\to V$ is said to be \textbf{linear} if for every simplex $\Delta\subset P$
\[  f\left(\sum_{v\in \ver(\Delta)} \lambda_v v \right)  = \sum_{v\in \ver(\Delta)} \lambda_v f(v), \]
where $0\leq \lambda_v \leq 1$ and $\sum_{v\in \ver(\Delta)} \lambda_v = 1$.
We denote the space of linear maps $f:P\to V$ by $\mathrm{Hom}(P,V)$.
When $V$ is a Banach space, the space $\mathrm{Hom}(P,V)$ is topologized as a product space $V^{\ver(P)}$.

\begin{lemma}  \label{lemma: preparations on linear maps}
Let $(V,\norm{\cdot})$ be a Banach space and $P$ a simplicial complex.
  \begin{enumerate}
    \item  If $f:P\to V$ is a linear map with $\diam f(P)\leq 2$ then for any $0<\varepsilon \leq 1$
   \[  \#(f(P),\norm{\cdot},\varepsilon) \leq C(P)\cdot (1/\varepsilon)^{\dim P}. \]
    Here the left-hand side is the minimum cardinality of open covers $\mathcal{U}$ of $f(P)$
    satisfying $\diam U < \varepsilon$ for all $U\in \mathcal{U}$ 
    (see the beginning of \S \ref{subsection: tame growth of covering numbers}).
    $C(P)$ is a positive constant depending only on $\dim P$ and the number of simplexes of $P$.
    
    \item Suppose $V$ is infinite dimensional. Then the set 
    \begin{equation} \label{eq: the set of injective linear maps}
       \{f\in \mathrm{Hom}(P,V)|\,  \text{$f$ is injective} \} 
    \end{equation}   
    is dense in $\mathrm{Hom}(P,V)$.

    \item  Let $(\mathcal{X},\mathbf{d})$ be a compact metric space and $\varepsilon, \delta>0$.
    Let $\pi:\mathcal{X}\to P$ be a continuous map satisfying $\diam\, \pi^{-1}(O_P(v)) < \varepsilon$ for all 
    $v\in \ver(P)$.  Let $f:\mathcal{X}\to V$ be a continuous map such that 
    \[   \mathbf{d}(x,y) < \varepsilon \Longrightarrow \norm{f(x)-f(y)} < \delta. \]
    Then there exists a linear map $g:P\to V$ satisfying 
    \[  \norm{f(x)-g(\pi(x))} < \delta \]
    for all $x\in \mathcal{X}$.
    Moreover if $f(\mathcal{X})$ is contained in the open unit ball $B_1^\circ(V)$ then we can assume $g(P)\subset B_1^\circ (V)$.
   \end{enumerate}
\end{lemma}

\begin{proof}
We sketch the proof.
See \cite[Lemma 5.3]{Lindenstrauss--Tsukamoto double VP} for the details.
The claim (1) is a direct calculation.
For (2), let $v_0,\dots, v_n$ be the vertexes of $P$. 
Since $V$ is infinite dimensional, the set 
\[ \{f\in \mathrm{Hom}(P,V)|\, f(v_0), \dots, f(v_n) \text{ are affinely independent}\} \]
is dense in $\mathrm{Hom}(P,V)$, and this is contained in (\ref{eq: the set of injective linear maps}).
For (3), let $v$ be a vertex of $P$.
Pick $x_v\in \pi^{-1}(O_P(v))$ and set $g(v) = f(x_v)$.
If $\pi^{-1}(O_P(v))=\emptyset$ then $g(v)$ may be an arbitrary point of $B_1^\circ (V)$.
We extend $g$ to a linear map from $P$ to $V$.
Then this map satisfies the requirements.
\end{proof}

Let $f:\mathcal{X}\to P$ be a continuous map from a topological space $\mathcal{X}$ to a simplicial complex $P$.
Recall that it is said to be essential if there is no proper subcomplex of $P$ containing $f(\mathcal{X})$
(see \S \ref{subsection: remarks on width dimension with potential}).
This is equivalent to the condition that for any simplex $\Delta\subset P$
\[  \bigcap_{v\in \ver(\Delta)} f^{-1}(O_P(v)) \neq \emptyset. \]

\begin{lemma}  \label{lemma: essential map}
Let $f:\mathcal{X}\to P$ be a continuous map from a topological space $\mathcal{X}$ to a simplicial complex $P$.
There exists a subcomplex $P'\subset P$ such that $f(\mathcal{X})\subset P'$ and $f:\mathcal{X}\to P'$ is 
essential.
\end{lemma}

\begin{proof}
Take the minimal subcomplex $P'\subset P$ containing $f(\mathcal{X})$.
\end{proof}

For two open covers $\mathcal{U}$ and $\mathcal{V}$ of $\mathcal{X}$, 
we say that $\mathcal{V}$ is a refinement of $\mathcal{U}$ (denoted by $\mathcal{U}\prec \mathcal{V}$)
if for every $V\in \mathcal{V}$ there exists $U\in \mathcal{U}$ containing $V$.

\begin{lemma} \label{lemma: preparation on simplicial map}
Let $\mathcal{X}$ be a topological space, $P$ and $Q$ simplicial complexes.
Let $\pi:\mathcal{X}\to P$ and $q_i:\mathcal{X}\to Q$ $(1 \leq i \leq N)$ be continuous maps.
We suppose that $\pi$ is essential and satisfies for all $1\leq i\leq N$ 
   \[   \left\{q_i^{-1}(O_Q(w))\right\}_{w\in \ver(Q)} \prec \left\{\pi^{-1}(O_P(v))\right\}_{v\in \ver(P)}
         \quad (\text{as open covers of $\mathcal{X}$}).  \]
Then there exist simplicial maps $h_i: P\to Q$ $(1\leq i \leq N)$ satisfying the following three conditions.
\begin{enumerate}
   \item For all $1\leq i\leq N$ and $x\in \mathcal{X}$ the two points $q_i(x)$ and $h_i(\pi(x))$ belong to the same simplex of 
           $Q$.
   \item Let $1\leq i\leq N$ and let $Q'\subset Q$ be a subcomplex.
           If a simplex $\Delta \subset P$ satisfies $\pi^{-1}\left(O_P(\Delta)\right) \subset q_i^{-1}\left(Q'\right)$ then 
           $h_i(\Delta) \subset Q'$.
   \item Let $\Delta\subset P$ be a simplex.
           If $q_i=q_j$ on $\pi^{-1}\left(O_P(\Delta)\right)$ then $h_i=h_j$ on $\Delta$.
\end{enumerate}
\end{lemma}

\begin{proof}
Let $v\in P$ be a vertex. We can choose $h_i(v)\in \ver(Q)$ such that 
\begin{itemize}
  \item $\pi^{-1}\left(O_P(v)\right) \subset q_i^{-1}\left(O_Q(h_i(v))\right)$.
  \item If $q_i=q_j$ on $\pi^{-1}\left(O_P(v)\right)$ then $h_i(v)=h_j(v)$.
\end{itemize}
Suppose $v_0, \dots, v_n$ span a simplex in $P$. Since $\pi$ is essential 
\[  \emptyset \neq \pi^{-1}\left(O_P(v_0)\cap \dots \cap O_P(v_n)\right)
     \subset q_i^{-1}\left(O_Q(h_i(v_0))\cap \dots \cap O_Q(h_i(v_n))\right). \]
Hence $O_Q(h_i(v_0))\cap \dots \cap O_Q(h_i(v_n)) \neq \emptyset$.     
This implies that $h_i(v_0), \dots, h_i(v_n)$ span a simplex in $Q$.
Hence $h_i$ can be extended to a simplicial map from $P$ to $Q$.
The maps $h_i$ satisfy the condition (3) from the above choice.
We can also check the conditions (1) and (2).
See \cite[Lemma 5.5]{Lindenstrauss--Tsukamoto double VP} for the details.
\end{proof}

Let $(\mathcal{X},\mathbf{d})$ be a compact metric space and $\mathcal{U}$ its open cover.
We define the \textbf{Lebesgue number} $LN(\mathcal{X},\mathbf{d},\mathcal{U})$ as the supremum of $\varepsilon>0$
such that if a subset $A\subset \mathcal{X}$ satisfies $\diam A < \varepsilon$ then 
there exists $U\in \mathcal{U}$ containing $A$.

\subsection{Dynamical tiling construction}  \label{subsection: dynamical tiling construction}

The purpose of this subsection is to define a ``dynamical decomposition'' of the real line, which was 
 first introduced in \cite[Section 4]{Gutman--Lindenstrauss--Tsukamoto}.
This will be the basis of the construction in the proof of Theorem \ref{theorem: dynamical PS theorem}.

Let $(\mathcal{X},T)$ be a dynamical system and $\psi:\mathcal{X}\to [0,1]$ a continuous function.
Take $x\in \mathcal{X}$. We consider 
\begin{equation} \label{eq: marker points}
    \left\{\left(a, \frac{1}{\psi(T^a x)}\right) \middle| \, a\in \mathbb{Z} \text{ with } \psi(T^a x) >0 \right\}. 
\end{equation}    
This is a discrete subset of the plane. We assume that (\ref{eq: marker points}) is nonempty for every $x\in \mathcal{X}$.
Namely for every $x\in \mathcal{X}$ there exists $a\in \mathbb{Z}$ with $\psi(T^a x)>0$.
Let $\mathbb{R}^2 = \bigcup_{a\in \mathbb{Z}} V_{\psi}(x,a)$ be the associated \textbf{Voronoi diagram}, where 
$V_\psi(x,a)$ is the (convex) set of $u\in \mathbb{R}^2$ satisfying 
\[  \left|u-\left(a,\frac{1}{\psi(T^a x)}\right)\right| \leq \left|u-\left(b, \frac{1}{\psi(T^b x)}\right)\right| \]
for any $b\in \mathbb{Z}$ with $\psi(T^b x) >0$.
(If $\psi(T^a x)=0$ then $V_\psi(x,a)$ is empty.)
We set 
\[   I_\psi(x,a) = V_\psi(x,a)\cap (\mathbb{R}\times \{0\}). \]
See Figure \ref{figure: dynamical tiling construction}.
(This is the same figure with the one in \cite[Subsection 5.4]{Lindenstrauss--Tsukamoto double VP}.)

\begin{figure}[h]
    \centering
    \includegraphics[width=5.0in]{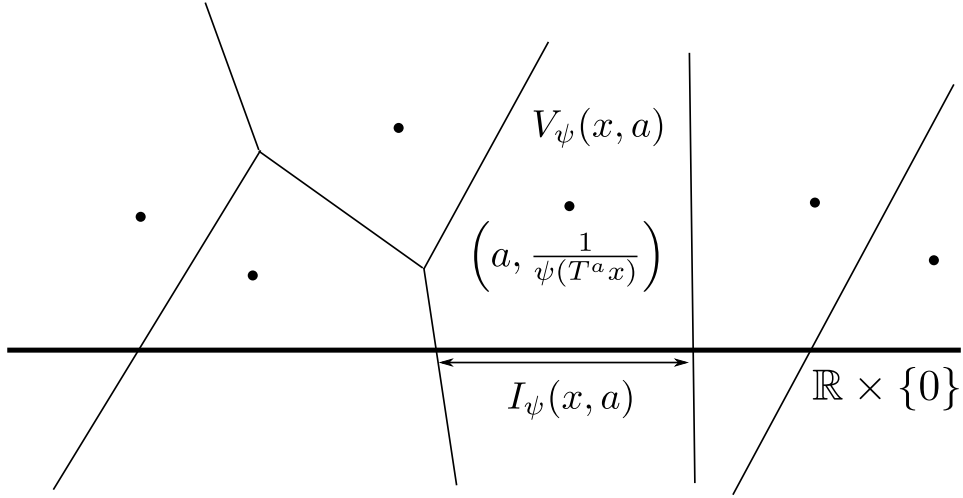}
    \caption{Dynamical tiling construction}   \label{figure: dynamical tiling construction}
\end{figure}

We naturally identify $\mathbb{R}\times \{0\}$ with $\mathbb{R}$.
This provides a decomposition of $\mathbb{R}$:
\[  \mathbb{R} = \bigcup_{a\in \mathbb{Z}} I_\psi(x,a). \]
We set 
\[  \partial_\psi(x) = \bigcup_{a\in \mathbb{Z}} \partial I_\psi(x,a) \subset \mathbb{R}, \]
where $\partial I_\psi(x,a)$ is the boundary of $I_\psi(x,a)$ (e.g. $\partial [0,1] = \{0,1\}$).
This construction is equivariant:
\[  I_\psi(T^n x, a) = -n + I_\psi(x,a+n), \quad
    \partial_\psi(T^n x) = -n + \partial_\psi(x). \]
Recall that a dynamical system $(\mathcal{X},T)$ is said to satisfy the marker property if for every $N>0$ there exists an open set 
$U\subset \mathcal{X}$ satisfying 
\begin{equation} \label{eq: marker property revisited}
    U\cap T^{-n} U = \emptyset \quad (1\leq n\leq N), \quad 
     \mathcal{X} = \bigcup_{n\in \mathbb{Z}} T^{-n} U. 
\end{equation}

\begin{lemma} \label{lemma: dynamical tiling}
Suppose $(\mathcal{X},T)$ satisfies the marker property.
Then for any $\varepsilon>0$ we can find a continuous function $\psi:\mathcal{X}\to [0,1]$ such that 
(\ref{eq: marker points}) is nonempty for every $x\in \mathcal{X}$ and that it satisfies that following two conditions.

   \begin{enumerate}
     \item There exists $M>0$ such that $I_\psi(x,a)\subset (a-M,a+M)$ for all $x\in \mathcal{X}$ and $a\in \mathbb{Z}$.
             The intervals $I_\psi(x,a)$ depend continuously on $x\in \mathcal{X}$, namely if $I_\psi(x,a)$ has positive length and 
             if $x_k\to x$ in $\mathcal{X}$ then $I_\psi(x_k,a)$ converges to $I_\psi(x,a)$ in the Hausdorff topology.
     \item The sets $\partial_\psi(x)$ are sufficiently ``sparse'' in the sense that 
            \[   \lim_{R\to \infty} \frac{\sup_{x\in \mathcal{X}} |\partial_\psi(x)\cap [0,R]|}{R} < \varepsilon. \]
            Here $|\partial_\psi(x)\cap [0,R]|$ is the cardinality of $\partial_\psi(x)\cap [0,R]$.
   \end{enumerate}
\end{lemma}

\begin{proof}
Take $N>1/\varepsilon$. There exists an open set $U\subset \mathcal{X}$ satisfying (\ref{eq: marker property revisited}).
We can find $M>N$ and a compact subset $K\subset U$ with $\mathcal{X} = \bigcup_{n=0}^{M-1}T^{-n} K$.
Let $\psi:\mathcal{X} \to [0,1]$ be a continuous function such that $\psi=1$ on $K$ and $\supp \psi \subset U$.
We can check that this function satisfies the requirements. See \cite[Lemma 5.10]{Lindenstrauss--Tsukamoto double VP}
for the details.
\end{proof}

\subsection{Proof of Theorem \ref{theorem: dynamical PS theorem}} \label{subsection: proof of dynamical PS theorem}

Theorem \ref{theorem: dynamical PS theorem} is contained in the following theorem.
For a topological space $\mathcal{X}$ and a Banach space $(V,\norm{\cdot})$
we denote by $C(\mathcal{X},V)$ the space of 
continuous maps $f:\mathcal{X}\to V$ endowed with the norm topology
(i.e. the topology given by the metric $\sup_{x\in \mathcal{X}}\norm{f(x)-g(x)}$).

\begin{theorem}[$\supset$ Theorem \ref{theorem: dynamical PS theorem}]
\label{theorem: dynamical PS theorem refined}
Let $(\mathcal{X},T)$ be a dynamical system with a continuous function 
$\varphi:\mathcal{X}\to \mathbb{R}$, and let 
$(V,\norm{\cdot})$ be an infinite dimensional Banach space.
Suppose $(\mathcal{X},T)$ has the marker property.
Then for a dense subset $f\in C(\mathcal{X},V)$, $f$ is a topological embedding and satisfies 
\[  \overline{\mdim}_{\mathrm{M}}(\mathcal{X},T,f^*\norm{\cdot}, \varphi) = \mdim(\mathcal{X},T,\varphi). \]
Here $f^*\norm{\cdot}$ is the metric $\norm{f(x)-f(y)}$ $(x,y\in \mathcal{X})$.
\end{theorem}

\begin{proof}
First we introduce some notations.
For a natural number $N$ we set $[N]=\{0,1,2,\dots, N-1\}$.
We define a norm on $V^N$ (the $n$-th power of $V$) by 
\[ \norm{(x_0,x_1,\dots, x_{N-1})}_N = \max\left(\norm{x_0},\norm{x_1},\dots, \norm{x_{N-1}}\right). \]
For simplicial complexes $P$ and $Q$ we define their \textbf{join} $P*Q$ as the quotient space of 
$[0,1]\times P\times Q$ by the equivalence relation 
\[ (0,p,q)\sim (0, p, q'), \quad 
    (1,p,q)\sim (1,p',q), \quad 
    (p,p'\in P, q, q'\in Q). \]
We denote the equivalence class of $(t,p,q)$ by $(1-t)p \oplus tq$.
We identify $P$ and $Q$ with $\{(0,p, *)|\, p\in P\}$ and $\{(1,*,q)|\, q\in Q\}$ in $P*Q$ respectively.
For a continuous map $f:\mathcal{X}\to V$ and $I\subset \mathbb{R}$ we define $\Phi_{f,I}:\mathcal{X}\to V^{I\cap \mathbb{Z}}$
by 
\[ \Phi_{f,I}(x) = \left(f(T^a x)\right)_{a\in I\cap \mathbb{Z}}. \]
For a natural number $R$ we set $\Phi_{f,R} := \Phi_{f,[R]}: \mathcal{X}\to V^R$.
We denote by $\Phi_{f,R}^*\norm{\cdot}_R$ the semi-metric
$\norm{\Phi_{f,R}(x)-\Phi_{f,R}(y)}_R$ on $\mathcal{X}$.
(It becomes a metric if $f$ is a topological embedding.)
For a semi-metric $\mathbf{d}'$ on $\mathcal{X}$ and $\varepsilon>0$ we define 
  \begin{equation*}
       \#\left(\mathcal{X},\mathbf{d}',\varphi,\varepsilon\right) = 
  \inf\left\{ \sum_{i=1}^n (1/\varepsilon)^{\sup_{U_i} \varphi} \middle|\,
   \parbox{3in}{\centering  $\mathcal{X} = U_1\cup \dots \cup U_n$ is an open cover with $\diam(U_i, \mathbf{d}') < \varepsilon$
                   for all $1\leq i\leq n$}\right\}, 
    \end{equation*}
where $\diam(U_i, \mathbf{d}')$ is the supremum of $\mathbf{d}'(x,y)$ over $x,y\in U_i$.
We fix a continuous function $\alpha:\mathbb{R} \to [0,1]$ such that 
$\alpha(t)= 1$ for $t\leq 1/2$ and $\alpha(t) = 0$ for $t\geq 3/4$.

We can assume $D=\mdim(\mathcal{X},T,\varphi)<\infty$.
Fix a metric $\mathbf{d}$ on $\mathcal{X}$.
Take an arbitrary continuous map $f:\mathcal{X}\to V$ and $\eta>0$.
Our purpose is to construct a topological embedding $f':\mathcal{X}\to V$ satisfying 
$\norm{f(x)-f'(x)}<\eta$ and $\overline{\mdim}_{\mathrm{M}}(\mathcal{X},T,(f')^*\norm{\cdot},\varphi) \leq D$.
(The reverse inequality $\overline{\mdim}_{\mathrm{M}}(\mathcal{X},T,(f')^*\norm{\cdot},\varphi) \geq D$
follows from Theorem \ref{theorem: mean Hausdorff dimension bounds mean dimension restated}.)
We may assume that $f(\mathcal{X})$ is contained in the open unit ball $B^\circ_1(V)$.
We will inductively construct the following data for $n\geq 1$.

\begin{data}  \label{data: dynamical PS theorem}
   \begin{enumerate}
      \item $1/2>\varepsilon_1>\varepsilon_2>\dots>0$ with $\varepsilon_{n+1}<\varepsilon_n/2$ and 
              $\eta/2>\delta_1>\delta_2>\dots>0$ with $\delta_{n+1}<\delta_n/2$.
      \item A natural number $N_n$.
      \item A continuous function $\psi_n:\mathcal{X}\to [0,1]$ such that 
              for every $x\in \mathcal{X}$ there exists $a\in \mathbb{Z}$ satisfying 
              $\psi_n(T^a x)>0$. We apply the dynamical tiling construction of 
              \S \ref{subsection: dynamical tiling construction} to $\psi_n$ and get the decomposition 
              $\mathbb{R} = \bigcup_{a\in \mathbb{Z}} I_{\psi_n}(x,a)$ for each $x\in \mathcal{X}$.   
      \item $(1/n)$-embeddings $\pi_n:(\mathcal{X},\mathbf{d}_{N_n})  \to P_n$ and 
              $\pi'_n:(\mathcal{X},\mathbf{d})\to Q_n$ with simplicial complexes $P_n$ and $Q_n$.     
      \item For each $\lambda\in [N_n]$, a linear map $g_{n,\lambda}:P_n\to B_1^\circ(V)$. 
      \item A linear map $g'_n:Q_n\to B^\circ_1(V)$.
   \end{enumerate}
\end{data}

We assume the following six conditions.

\begin{condition} \label{condition: dynamical PS theorem}
    \begin{enumerate}
      \item For each $\lambda\in [N_n]$, the map $g_{n,\lambda}*g'_n: P_n*Q_n\to B_1^\circ(V)$ is injective.
              For $\lambda_1\neq \lambda_2$
              \[   g_{n,\lambda_1}*g'_n(P_n*Q_n) \cap g_{n,\lambda_2}*g'_n(P_n*Q_n) = g_n(Q_n). \]
      \item Set $g_n=(g_{n,0},g_{n,1},\dots,g_{n,N_n-1}):P_n\to V^{N_n}$. 
              We assume that $\pi_n$ is essential and
              \[  \sum_{\Delta\subset P_n} 
                    \left(\frac{1}{\varepsilon}\right)^{\sup_{\pi_n^{-1}\left(O_{P_n}(\Delta)\right)} S_{N_n}\varphi}
                   \#(g_n(\Delta), \norm{\cdot}_{N_n}, \varepsilon)
                   < \left(\frac{1}{\varepsilon}\right)^{(D+\frac{3}{n})N_n}, 
                   \quad (0<\varepsilon \leq \varepsilon_n).   \]       
              Here $\Delta$ runs over simplexes of $P_n$.    
              Since $\pi_n$ is essential, $\pi_n^{-1}\left(O_{P_n}(\Delta)\right)$ is non-empty for every $\Delta\subset P_n$.  
      \item For $0<\varepsilon \leq \varepsilon_{n-1}$ $(n\geq 2)$, 
              \[  \#\left(\mathcal{X},(g_n\circ\pi_n)^*\norm{\cdot}_{N_n},S_{N_n}\varphi,\varepsilon\right)
                   < 2^{N_n}\left(\frac{1}{\varepsilon}\right)^{\left(D+\frac{4}{n-1}\right)N_n}. \]      
              Here $(g_n\circ \pi_n)^*\norm{\cdot}_{N_n}$ is the semi-metric $\norm{g_n(\pi_n(x))-g_n(\pi_n(y))}$ on $\mathcal{X}$.       
              Notice that the condition (2) above is stronger than this over the region $0<\varepsilon \leq \varepsilon_n$.
              The point is that the condition (3) covers the region $\varepsilon_n<\varepsilon\leq \varepsilon_{n-1}$.
      \item There exists $M_n>0$ such that $I_{\psi_n}(x,a)\subset (a-M_n,a+M_n)$ for all $x\in \mathcal{X}$ and $a\in \mathbb{Z}$.     
               We take $C_n \geq 1$ satisfying 
               \begin{equation} \label{eq: size of P_n*Q_n}
                    \#\left(\bigcup_{\lambda\in [N_n]} g_{n,\lambda}*g'_n(P_n*Q_n), \norm{\cdot},\varepsilon\right)
                    < \left(\frac{1}{\varepsilon}\right)^{C_n} \quad (0<\varepsilon \leq 1/2). 
               \end{equation}       
              Then we assume 
              \[  \lim_{R\to \infty} \frac{\sup_{x\in \mathcal{X}} |\partial_{\psi_n}(x)\cap [0,R]|}{R}
                   < \frac{1}{2nN_n(C_n+\norm{\varphi}_\infty)},  \]
               where $\norm{\varphi}_\infty = \max_{x\in \mathcal{X}} |\varphi(x)|$.    
      \item We define a continuous map $f_n:\mathcal{X}\to B^\circ_1(V)$ as follows.
              Let $x\in \mathcal{X}$.  Take $a\in \mathbb{Z}$ with $0\in I_{\psi_n}(x,a)$, and 
              take $b\in \mathbb{Z}$ satisfying $b\equiv a (\mathrm{mod} N_n)$ and $0\in b+[N_n]$.         
              We set 
              \begin{equation} \label{eq: continuous map f_n in dynamical PS theorem}
                  f_n(x) =  \left\{1-\alpha\left(\dist(0,\partial_{\psi_n}(x))\right)\right\} g_{n,-b}\left(\pi_n(T^b x)\right) 
                             + \alpha\left(\dist(0,\partial_{\psi_n}(x))\right) g'_n(\pi'_n(x)),                       
              \end{equation}                  
              where $\dist(0,\partial_{\psi_n}(x)) = \min_{t\in \partial_{\psi_n}(x)} |t|$.
              The condition (1) above implies that the map $f_n$ is a $(1/n)$-embedding with respect to $\mathbf{d}$.
              Then we assume that if a continuous map $f':\mathcal{X}\to V$ satisfies 
              $\norm{f'(x)-f_n(x)} < \delta_n$ for all $x\in \mathcal{X}$ then $f'$ is a $(1/n)$-embedding with respect to $\mathbf{d}$.
      \item For all $x\in \mathcal{X}$
              \[  \norm{f(x)-f_1(x)} < \frac{\eta}{2}, \quad 
                   \norm{f_n(x)-f_{n+1}(x)} < \min\left(\frac{\varepsilon_n}{8}, \frac{\delta_n}{2}\right). \]          
    \end{enumerate}
\end{condition}

Suppose we have constructed the above data.
We define a continuous map $f':\mathcal{X}\to V$ by $f'(x) = \lim_{n\to \infty} f_n(x)$.
(The convergence follows from the condition (6) above.) 
It satisfies $\norm{f'(x)-f(x)}<\eta$ and 
$\norm{f'(x)-f_n(x)} < \min(\varepsilon_n/4,\delta_n)$ for all $n\geq 1$.
Then the condition (5) implies that $f'$ is a $(1/n)$-embedding with respect to $\mathbf{d}$ for all $n\geq 1$, which means that 
$f'$ is a topological embedding.
We estimate 
\[  \overline{\mdim}_{\mathrm{M}}(\mathcal{X},T,(f')^*\norm{\cdot},\varphi) 
     = \limsup_{\varepsilon\to 0} 
     \left\{\left(\lim_{R\to \infty} \frac{\log\#(\mathcal{X},\Phi_{f',R}^*\norm{\cdot}_R, S_R\varphi, \varepsilon)}{R}\right)
     \middle/ \log(1/\varepsilon)\right\}. \]
Let $0<\varepsilon <\varepsilon_1$. Take $n>1$ with $\varepsilon_n\leq \varepsilon < \varepsilon_{n-1}$.     
From $\norm{f'(x)-f_n(x)}<\varepsilon_n/4$
\begin{equation*}
     \#(\mathcal{X},\Phi_{f',R}^*\norm{\cdot}_R, S_R\varphi,\varepsilon)  \leq
      \#\left(\mathcal{X},\Phi_{f_n,R}^*\norm{\cdot}_R, S_R\varphi, \varepsilon -\frac{\varepsilon_n}{2}\right) 
     \leq  \#\left(\mathcal{X},\Phi_{f_n,R}^*\norm{\cdot}_R, S_R\varphi, \frac{\varepsilon}{2}\right).
\end{equation*}    
From Claim \ref{claim: estimate of covering number in PS theorem 1} below, 
\[ \lim_{R\to \infty} \frac{\log \#(\mathcal{X},\Phi_{f',R}^*\norm{\cdot},S_R\varphi,\varepsilon)}{R} 
   \leq 2+ \left(D+\frac{4}{n-1}+\frac{1}{n} \right) \log\left(\frac{2}{\varepsilon}\right). \]
Since $n\to \infty$ as $\varepsilon \to 0$,
this proves $\overline{\mdim}_{\mathrm{M}}(\mathcal{X},T,(f')^*\norm{\cdot},\varphi) \leq D$.

\begin{claim} \label{claim: estimate of covering number in PS theorem 1}
      Let $0<\varepsilon <\varepsilon_{n-1}$ $(n\geq 2)$.
      If $R$ is a sufficiently large natural number then
      \[  \#\left(\mathcal{X},\Phi_{f_n,R}^*\norm{\cdot}_R,S_R\varphi,\varepsilon\right) 
              \leq 4^R \left(\frac{1}{\varepsilon}\right)^{\left(D+\frac{4}{n-1}\right)R+\frac{R}{n}} \]
\end{claim}

\begin{proof}
Let $x\in \mathcal{X}$.
A discrete interval $J=[b,b+N_n)\cap \mathbb{Z}$ of length $N_n$ $(b\in \mathbb{Z})$ is said to be \textbf{good for $x$}
if there exists $a\in \mathbb{Z}$ such that $b\equiv a (\mathrm{mod} N_n)$ and 
$[b-1, b+N_n]\subset I_{\psi_n}(x,a)$. 
If $J$ is good for $x$ then 
\[  \Phi_{f_n,J}(x) = g_n\left(\pi_n(T^b x)\right) \in g_n(P_n). \]

We denote by $\mathcal{J}_x$ the union of $J\subset [R]$ which are good for $x$.
For a subset $\mathcal{J}\subset [R]$ we define $\mathcal{X}_{\mathcal{J}}$ as the set of 
$x\in \mathcal{X}$ satisfying $\mathcal{J}_x=\mathcal{J}$.
The set $\mathcal{X}_{\mathcal{J}}$ may be empty. 
If it is non-empty, then from Condition \ref{condition: dynamical PS theorem} (3)
\begin{equation} \label{eq: estimate of covering number of X_J in PS theorem 1}
   \#\left(\mathcal{X}_{\mathcal{J}},\Phi_{f_n,R}^*\norm{\cdot}_R, S_R\varphi,\varepsilon\right) 
   \leq  \left\{2^{N_n}\left(\frac{1}{\varepsilon}\right)^{\left(D+\frac{4}{n-1}\right)N_n}
   \right\}^{|\mathcal{J}|/N_n} 
    \cdot  \left(\frac{1}{\varepsilon}\right)^{(C_n+\norm{\varphi}_\infty) \left|[R]\setminus \mathcal{J}\right|}.
\end{equation}   
Here $C_n$ is the positive constant introduced in (\ref{eq: size of P_n*Q_n}).
We have $|\mathcal{J}|\leq R$ and 
\[  \left|[R]\setminus \mathcal{J}\right| \leq 2N_n \sup_{x\in \mathcal{X}} \left|\partial_{\psi_n}(x)\cap [0,R] \right| + 2N_n. \]
The second term ``$+2N_n$'' in the right-hand side is the edge effect.
From Condition \ref{condition: dynamical PS theorem} (4), for sufficiently large $R$
\[  (C_n+\norm{\varphi}_\infty) \left|[R]\setminus \mathcal{J}\right| < \frac{R}{n}. \]
Then the quantity (\ref{eq: estimate of covering number of X_J in PS theorem 1}) is bounded by 
\[  2^{R} \left(\frac{1}{\varepsilon}\right)^{\left(D+\frac{4}{n-1}\right)R  + \frac{R}{n}}.  \]
The number of the choices of $\mathcal{J}\subset [R]$ is bounded by $2^{R}$.
Thus 
\[    \#\left(\mathcal{X},\Phi_{f_n,R}^*\norm{\cdot}_R, S_R\varphi,\varepsilon\right) 
      \leq 4^{R}  \left(\frac{1}{\varepsilon}\right)^{\left(D+\frac{4}{n-1}\right)R + \frac{R}{n}}.  \]
\end{proof}

\textbf{Induction: Step 1.}
Now we start to construct the data.
First we construct them for $n=1$.
Take $0<\tau_1<1$ such that 
\[  \mathbf{d}(x,y) < \tau_1 \Longrightarrow \norm{f(x)-f(y)} < \frac{\eta}{2}, \quad 
     |\varphi(x)-\varphi(y)| < 1. \]
From $\mdim(\mathcal{X},T,\varphi) =D$, we can find $N_1>0$, a simplicial complex $P_1$ and 
a $\tau_1$-embedding $\pi_1:(\mathcal{X},\mathbf{d}_{N_1})\to P_1$ such that 
$\dim_{\pi_1(x)}P_1+S_{N_1}\varphi(x) < N_1(D+1)$ for all $x\in \mathcal{X}$.
We also take a simplicial complex $Q_1$ and a $\tau_1$-embedding 
$\pi'_1:(\mathcal{X},\mathbf{d})\to Q_1$.
By subdividing $P_1$ and $Q_1$ if necessary, we can assume that 
for all simplexes $\Delta\subset P_1$ and all $w\in \ver(Q_1)$
\[  \diam \left(\pi^{-1}_1\left(O_{P_1}(\Delta)\right), \mathbf{d}_{N_1}\right) < \tau_1, \quad 
     \diam \left((\pi_1')^{-1}(O_{Q_1}(w)),\mathbf{d}\right) < \tau_1. \]
Moreover by Lemma \ref{lemma: essential map} we can assume that $\pi_1$ is essential.

By Lemma \ref{lemma: preparations on linear maps} (3) there exist linear maps 
$g_{1,\lambda}:P_1\to B_1^\circ(V)$ $(\lambda\in [N_1])$ and 
$g'_1:Q_1\to B^\circ_1(V)$ satisfying 
\begin{equation} \label{eq: discrepancy of perturbation in induction step 1}
 \norm{f(T^\lambda x)-g_{1,\lambda}(\pi_1(x))} < \frac{\eta}{2}, \quad 
    \norm{f(x)-g'_1(\pi_1(x))} < \frac{\eta}{2}. 
\end{equation}    
We slightly perturb $g_{1,\lambda}$ and $g'_1$ (if necessary) by Lemma \ref{lemma: preparations on linear maps}
(2) so that they satisfy Condition \ref{condition: dynamical PS theorem} (1).

By Lemma \ref{lemma: preparations on linear maps} (1), we can choose $0<\varepsilon_1<1/2$ such that 
for any $0<\varepsilon \leq \varepsilon_1$ and simplex $\Delta\subset P_1$
\[  \#\left(g_1(\Delta),\norm{\cdot}_{N_1}, \varepsilon\right) < 
    \frac{1}{(\text{Number of simplexes of $P_1$})} \left(\frac{1}{\varepsilon}\right)^{\dim \Delta + 1}. \]
Let $\Delta\subset P_1$ be a simplex.
Since $\pi_1$ is essential, we can find a point $x\in \pi_1^{-1}\left(O_{P_1}(\Delta)\right)$ with 
$\dim \Delta \leq \dim_{\pi_1(x)} P_1$.
From the choice of $\tau_1$ 
\[ \sup_{\pi_1^{-1}\left(O_{P_1}(\Delta)\right)} S_{N_1}\varphi \leq S_{N_1}\varphi(x) + N_1. \]
 Hence for $0<\varepsilon \leq \varepsilon_1$
\begin{equation*} 
   \begin{split}
   & \left(\frac{1}{\varepsilon}\right)^{\sup_{\pi_1^{-1}\left(O_{P_1}(\Delta)\right)} S_{N_1}\varphi}
   \#\left(g_1(\Delta),\norm{\cdot}_{N_1},\varepsilon\right)  \\ 
   &  <    \frac{1}{(\text{Number of simplexes of $P_1$})} 
         \left(\frac{1}{\varepsilon}\right)^{\dim \Delta +S_{N_1}\varphi(x) + N_1+ 1} \\
   & \leq         \frac{1}{(\text{Number of simplexes of $P_1$})} 
         \left(\frac{1}{\varepsilon}\right)^{\dim_{\pi_1(x)} P_1 +S_{N_1}\varphi(x) + N_1+ 1}.
   \end{split}      
\end{equation*}
From $\dim_{\pi_1(x)} P_1 + S_{N_1}\varphi(x) < N_1(D+1)$,
this is bounded by 
\begin{equation*}
   \begin{split}
    & \frac{1}{(\text{Number of simplexes of $P_1$})}  \left(\frac{1}{\varepsilon}\right)^{N_1(D+1)+N_1+1} \\
    &     \leq   \frac{1}{(\text{Number of simplexes of $P_1$})}   \left(\frac{1}{\varepsilon}\right)^{N_1(D+3)}. 
  \end{split}
\end{equation*}  
This shows Condition \ref{condition: dynamical PS theorem} (2):
\[  \sum_{\Delta\subset P_1}
      \left(\frac{1}{\varepsilon}\right)^{\sup_{\pi_1^{-1}\left(O_{P_1}(\Delta)\right)} S_{N_1}\varphi}
   \#\left(g_1(\Delta),\norm{\cdot}_{N_1},\varepsilon\right)
   <   \left(\frac{1}{\varepsilon}\right)^{N_1(D+3)}. \]
Condition \ref{condition: dynamical PS theorem} (3) is empty for $n=1$.

By Lemma \ref{lemma: dynamical tiling} we can choose a continuous function $\psi_1:\mathcal{X}\to [0,1]$
satisfying Condition \ref{condition: dynamical PS theorem} (4).

The continuous map $f_1:\mathcal{X}\to V$ defined in (\ref{eq: continuous map f_n in dynamical PS theorem})
is a $1$-embedding. Since ``$1$-embedding'' is an open condition, we can choose $0<\delta_1<\eta/2$
such that any continuous map $f':\mathcal{X}\to V$ with $\norm{f_1(x)-f'(x)} < \delta_1$ is also 
a $1$-embedding. This establishes Condition \ref{condition: dynamical PS theorem} (5). 

From (\ref{eq: discrepancy of perturbation in induction step 1}) we get Condition \ref{condition: dynamical PS theorem} (6):
\[  \norm{f(x)- f_1(x)} < \frac{\eta}{2}. \]
We have completed the construction of the data for $n=1$.

\vspace{0.2cm}

\textbf{Induction: Step $n$ $\Rightarrow$ Step $n+1$.}
Suppose we have constructed the data for $n$.
We will construct the data for $n+1$.

We subdivide the join $P_n*Q_n$ sufficiently fine (denoted by $\overline{P_n*Q_n}$)
such that for all simplexes $\Delta\subset \overline{P_n*Q_n}$ and all $\lambda\in [N_n]$
\begin{equation} \label{eq: diameter of g_{n,lambda}*g'_n(Delta)}
    \diam \left(g_{n,\lambda}*g'_n(\Delta), \norm{\cdot}\right) < \min\left(\frac{\varepsilon_n}{8}, \frac{\delta_n}{2}\right). 
\end{equation}
Since $P_n$ and $Q_n$ are (naturally) subcomplexes of $P_n*Q_n$,
this also introduces subdivisions of $P_n$ and $Q_n$ (denoted by $\overline{P_n}$ and $\overline{Q_n}$).

We define a continuous map 
$q_n: \mathcal{X}\to \overline{P_n*Q_n}$ as follows.
Let $x\in \mathcal{X}$.
Take $a,b\in \mathbb{Z}$ such that $0\in I_{\psi_n}(x,a)$, $a\equiv b  (\mathrm{mod} N_n)$ and 
$0\in b+[N_n]$.  
Then we set 
\[  q_n(x) =  \left\{1-\alpha\left(\dist(0,\partial_{\psi_n}(x))\right)\right\} \pi_n(T^b x)
                             \oplus \alpha\left(\dist(0,\partial_{\psi_n}(x))\right) \pi'_n(x). \] 
(This is a point in the join $P_n*Q_n$. We identify it with the point of the subdivision $\overline{P_n*Q_n}$.)
We have 
\begin{equation}  \label{eq: express f_n by q_n}
 f_n(x) = g_{n,-b}*g'_n(q_n(x)).
\end{equation}

Take $0<\tau_{n+1}<1/(n+1)$ satisfying the following four conditions.

\begin{enumerate}
  \renewcommand{\theenumi}{\roman{enumi}}
    \item If $\mathbf{d}(x,y) < \tau_{n+1}$ then $\norm{f_n(x)-f_n(y)} < \min(\varepsilon_n/8,\delta_n/2)$.
    \item If $\mathbf{d}(x,y) < \tau_{n+1}$ then $|\varphi(x)-\varphi(y)|<\frac{1}{n+1}$.
    \item If $\mathbf{d}(x,y)<\tau_{n+1}$ then the decompositions $\mathbb{R}= \bigcup_{a\in \mathbb{Z}}I_{\psi_n}(x,a)$ and 
            $\mathbb{R}= \bigcup_{a\in \mathbb{Z}}I_{\psi_n}(y,a)$ are ``close'' in the following two senses.
            \begin{itemize}
              \item 
                     \[  \left|\dist\left(0,\partial_{\psi_n}(x)\right) - \dist\left(0,\partial_{\psi_n}(y)\right)\right| < \frac{1}{4}. \]
              \item If $(-1/4,1/4)\subset I_{\psi_n}(x,a)$ then 
                      $0$ is an interior point of $I_{\psi_n}(y,a)$.
            \end{itemize}
    \item Consider the open cover $\left\{q_n^{-1}\left(O_{\overline{P_n*Q_n}}(v)\right)\right\}_{v\in \ver\left(\overline{P_n*Q_n}\right)}$
            of $\mathcal{X}$.
            The number $\tau_{n+1}$ is smaller than its Lebesgue number:           
            \[ \tau_{n+1} < 
            LN\left(\mathcal{X},\mathbf{d},\left\{q_n^{-1}
            \left(O_{\overline{P_n*Q_n}}(v)\right)\right\}_{v\in \ver\left(\overline{P_n*Q_n}\right)}\right).  \]
\end{enumerate}

Take a $\tau_{n+1}$-embedding 
$\pi'_{n+1}: (\mathcal{X},\mathbf{d})\to Q_{n+1}$ with a simplicial complex $Q_{n+1}$.
By subdividing it (if necessary), we can assume that $\diam \left((\pi'_{n+1})^{-1}\left(O_{Q_{n+1}}(w)\right),\mathbf{d}\right) < \tau_{n+1}$
for all $w\in \ver(Q_{n+1})$.
By Lemma \ref{lemma: preparations on linear maps} (3) there exists a linear map 
$\tilde{g}'_{n+1}:Q_{n+1}\to B^\circ_1(V)$ satisfying 
\begin{equation} \label{eq: discrepancy of tilde{g}_{n+1}}
  \norm{\tilde{g}'_{n+1}(\pi'_{n+1}(x))-f_n(x)} < \min\left(\frac{\varepsilon_n}{8},\frac{\delta_n}{2}\right).
\end{equation}

Take $N_{n+1}> N_n$ satisfying the following two conditions.

\begin{enumerate}
  \renewcommand{\theenumi}{\alph{enumi}}
   \item There exists a $\tau_{n+1}$-embedding $\pi_{n+1}:(\mathcal{X},\mathbf{d}_{N_{n+1}})\to P_{n+1}$
            with a simplicial complex $P_{n+1}$ such that for all $x\in \mathcal{X}$
            \begin{equation}  \label{eq: local dimension bound on P_{n+1}}
                \frac{\dim_{\pi_{n+1}(x)} P_{n+1} + S_{N_{n+1}}\varphi(x)}{N_{n+1}} < D + \frac{1}{n+1}. 
            \end{equation}    
   \item 
          \[ \frac{1+\sup_{x\in \mathcal{X}}\left|\partial_{\psi_n}(x)\cap [0,N_{n+1}]\right|}{N_{n+1}} 
              < \frac{1}{2n N_n (C_n+\norm{\varphi}_\infty)}, \]
          where $C_n$ is the positive constant introduced in (\ref{eq: size of P_n*Q_n}).
\end{enumerate}  

 By subdividing $P_{n+1}$ if necessary, we can assume that for any simplexes $\Delta, \Delta'\subset P_{n+1}$ with 
 $\Delta\cap \Delta'\neq \emptyset$
 \begin{equation} \label{eq: diameter of the preimage of open star of P_{n+1}}
      \diam\left(\pi_{n+1}^{-1}\left(O_{P_{n+1}}(\Delta)\right)\cup \pi_{n+1}^{-1}\left(O_{P_{n+1}}(\Delta')\right), \mathbf{d}_{N_{n+1}}\right) 
        < \tau_{n+1}.
 \end{equation}
  Moreover by Lemma \ref{lemma: essential map} we can assume that 
  $\pi_{n+1}$ is essential.

We apply Lemma \ref{lemma: preparation on simplicial map} to 
$\pi_{n+1}:\mathcal{X}\to P_{n+1}$ and $q_n\circ T^\lambda: \mathcal{X}\to \overline{P_n*Q_n}$ $(\lambda\in [N_{n+1}])$
with $P=P_{n+1}$, $Q=\overline{P_n*Q_n}$, $N=N_{n+1}$, and $Q'=\overline{P_n} \text{ or } \overline{Q_n}$.
The assumption of Lemma \ref{lemma: preparation on simplicial map} is satisfied by 
the condition (iv) of the choice of $\tau_{n+1}$.
Then we get simplicial maps $h_\lambda:P_{n+1}\to \overline{P_n*Q_n}$ $(\lambda\in [N_{n+1}])$ satisfying 
the following three conditions.

\begin{enumerate}
  \renewcommand{\theenumi}{\Alph{enumi}}
   \item For every $\lambda\in [N_{n+1}]$ and $x\in \mathcal{X}$, the two points $h_\lambda(\pi_{n+1}(x))$ and 
           $q_n(T^\lambda x)$ belong to the same simplex of $\overline{P_n*Q_n}$.
   \item Let $\lambda\in [N_{n+1}]$ and $\Delta\subset P_{n+1}$ be a simplex.
           If $\pi_{n+1}^{-1}\left(O_{P_{n+1}}(\Delta)\right)\subset T^{-\lambda} q_n^{-1}\left(\overline{P_n}\right)$
           then $h_\lambda(\Delta)\subset \overline{P_n}$.
           Similarly, if $\pi_{n+1}^{-1}\left(O_{P_{n+1}}(\Delta)\right)\subset T^{-\lambda} q_n^{-1}\left(\overline{Q_n}\right)$
           then $h_\lambda(\Delta)\subset \overline{Q_n}$.
   \item Let $\lambda, \lambda'\in [N_{n+1}]$ and $\Delta\subset P_{n+1}$ be a simplex.
           If $q_n\circ T^\lambda = q_n\circ T^{\lambda'}$ on $\pi_{n+1}^{-1}\left(O_{P_{n+1}}(\Delta)\right)$then 
           $h_\lambda = h_{\lambda'}$ on $\Delta$.
\end{enumerate}

We define a linear map $\tilde{g}_{n+1,\lambda}:P_{n+1}\to B^\circ_1(V)$ for each $\lambda\in [N_{n+1}]$ as follows.
Let $\Delta\subset P_{n+1}$ be a simplex.
Since $\pi_{n+1}:\mathcal{X}\to P_{n+1}$ is essential, we can find a point 
$x\in \pi_{n+1}^{-1}\left(O_{P_{n+1}}(\Delta)\right)$.
Take $a,b\in \mathbb{Z}$ with $\lambda\in I_{\psi_n}(x,a)$, $b\equiv a (\mathrm{mod} N_n)$ and 
$\lambda\in b+[N_n]$. 
Set 
\[  \tilde{g}_{n+1,\lambda}(u) = g_{n,\lambda-b}*g'_n(h_\lambda(u)) \quad 
     (u\in \Delta). \]
(See Claim \ref{claim: tilde{g}_{n,lambda} is well-defined} below for the well-definedness.)
As in (\ref{eq: express f_n by q_n}) we have 
\[  f_n(T^\lambda x) =  g_{n,\lambda-b}*g'_n\left(q_n(T^\lambda x)\right). \]
From (\ref{eq: diameter of g_{n,lambda}*g'_n(Delta)}) and the condition (A) of the choice of $h_\lambda$ 
\begin{equation} \label{eq: f_n T^lambda and tilde{g}_{n+1,lambda} is close}
    \norm{\tilde{g}_{n+1,\lambda}(\pi_{n+1}(x)) - f_n(T^\lambda x)} < \min\left(\frac{\varepsilon_n}{8},\frac{\delta_n}{2}\right).
\end{equation}

\begin{claim} \label{claim: tilde{g}_{n,lambda} is well-defined}
The above construction of $\tilde{g}_{n+1,\lambda}$ is independent of the various choices.
Namely, let $\Delta'\subset P_{n+1}$ be another simplex with $\Delta\cap \Delta'\neq \emptyset$.
Let $x'\in \pi_{n+1}^{-1}\left(O_{P_{n+1}}(\Delta')\right)$ and
take $a', b'\in \mathbb{Z}$ such that $\lambda\in I_{\psi_n}(x',a')$, $b' \equiv a' (\mathrm{mod} N_n)$
and $\lambda\in b'+[N_n]$. Then for $u\in \Delta\cap \Delta'$
\[   g_{n,\lambda-b'}*g'_n(h_\lambda(u)) = g_{n,\lambda-b}*g'_n(h_\lambda(u)). \]
\end{claim}

\begin{proof}
First suppose $\dist(\lambda, \partial_{\psi_n}(x)) > 1/4$.
From (\ref{eq: diameter of the preimage of open star of P_{n+1}}), 
we have $\mathbf{d}(T^\lambda x, T^\lambda x') < \tau_{n+1}$.
From the condition (iii) of the choice of $\tau_{n+1}$, $\lambda$ is an interior point of $I_{\psi_n}(x',a)$.
So $a=a'$ and $b=b'$. Then 
\[   g_{n,\lambda-b'}*g'_n(h_\lambda(u)) = g_{n,\lambda-b}*g'_n(h_\lambda(u)). \]
Next suppose $\dist(\lambda,\partial_{\psi_n}(x)) \leq 1/4$.
Let $y\in \pi_{n+1}^{-1}\left(O_{P_{n+1}}(\Delta)\right)\cup \pi_{n+1}^{-1}\left(O_{P_{n+1}}(\Delta')\right)$
be an arbitrary point.
From  $\mathbf{d}(T^\lambda x, T^\lambda y) < \tau_{n+1}$
and the condition (iii) of the choice of $\tau_{n+1}$, we have 
$\dist(\lambda,\partial_{\psi_n}(y)) < 1/2$. Then 
\[  q_n(T^\lambda y) = \pi'_n(T^\lambda y) \in \overline{Q_n}. \]
Since  $y\in \pi_{n+1}^{-1}\left(O_{P_{n+1}}(\Delta)\right)\cup \pi_{n+1}^{-1}\left(O_{P_{n+1}}(\Delta')\right)$ is arbitrary,
\[   \pi_{n+1}^{-1}\left(O_{P_{n+1}}(\Delta)\right)\cup \pi_{n+1}^{-1}\left(O_{P_{n+1}}(\Delta')\right)
      \subset T^{-\lambda} q_n^{-1}\left(\overline{Q_n}\right). \]
From the condition (B) of the choice of $h_\lambda$,
\[  h_\lambda(\Delta) \cup h_\lambda(\Delta') \subset \overline{Q_n}. \]
Then 
\[    g_{n,\lambda-b'}*g'_n(h_\lambda(u)) = g'_n(h_\lambda(u)) = g_{n,\lambda-b}*g'_n(h_\lambda(u)). \]
\end{proof}

\begin{claim} \label{claim: covering number with respect to tilde{g}_{n+1}}
Set $\tilde{g}_{n+1} = (\tilde{g}_{n+1,0},\dots, \tilde{g}_{n+1,N_{n+1}-1}): P_{n+1}\to V^{N_{n+1}}$.
For $0<\varepsilon \leq \varepsilon_n$
\[  \#\left(\mathcal{X}, (\tilde{g}_{n+1}\circ \pi_{n+1})^*\norm{\cdot}_{N_{n+1}}, S_{N_{n+1}}\varphi, \varepsilon\right)
     < 2^{N_{n+1}} \left(\frac{1}{\varepsilon}\right)^{\left(D+\frac{4}{n}\right) N_{n+1}}. \]
\end{claim}

\begin{proof}
This is close to the proof of Claim \ref{claim: estimate of covering number in PS theorem 1}. But it is a bit more involved.
Let $x\in \mathcal{X}$.
We say that a discrete interval $J=[b,b+N_n)\cap \mathbb{Z}$ of length $N_n$ $(b\in \mathbb{Z})$ is \textbf{good for $x$}
if $J\subset [N_{n+1}]$ and there exists $a\in \mathbb{Z}$ satisfying $b\equiv a (\mathrm{mod} N_n)$ and 
$[b-1,b+N_n] \subset I_{\psi_n}(x,a)$.

Suppose $J=[b,b+N_n)\cap \mathbb{Z}$ is good for $x\in \mathcal{X}$.
Take a simplex $\Delta\subset P_{n+1}$ containing $\pi_{n+1}(x)$.
Let $y\in \pi_{n+1}^{-1}\left(O_{P_{n+1}}(\Delta)\right)$ be an arbitrary point.
From (\ref{eq: diameter of the preimage of open star of P_{n+1}}) we have 
$\mathbf{d}_{N_{n+1}}(x,y) < \tau_{n+1}$.
From the condition (iii) of the choice of $\tau_{n+1}$,
\[ \left[b-\frac{3}{4}, b+N_n-\frac{1}{4}\right] \subset I_{\psi_n}(y,a).  \]
Then for all $\lambda\in J$
\[  q_n(T^\lambda y) = q_n(T^b y) = \pi_n(T^b y) \in \overline{P_n}.  \]
From the conditions (B) and (C) of the choice of $h_\lambda$, 
\[  h_b(\Delta) \subset \overline{P_n}, \quad 
    h_\lambda = h_b \text{ on $\Delta$ for $\lambda\in J$}. \] 
Then 
\[  \left(\tilde{g}_{n+1,\lambda}(\pi_{n+1}(x))\right)_{\lambda\in J}
     = g_n\left(h_b(\pi_{n+1}(x))\right). \]
Moreover it follows from the condition (A) of the choice of $h_\lambda$ that 
$h_b(\pi_{n+1}(x))$ and $q_n(T^b x) = \pi_n(T^b x)$ belong to the same simplex of $\overline{P_n}$.

For $x\in \mathcal{X}$ we denote by $\mathcal{J}_x$ the union of the intervals $J \subset [N_{n+1}]$ good for $x$.
For a subset $\mathcal{J}\subset [N_{n+1}]$ we define $\mathcal{X}_\mathcal{J}$ as the set of 
$x\in \mathcal{X}$ with $\mathcal{J}_x = \mathcal{J}$.
The set $\mathcal{X}_{\mathcal{J}}$ may be empty.
If it is non-empty, then from Condition \ref{condition: dynamical PS theorem} (2)
\begin{equation} \label{eq: covering number of X_J in the (n+1)-th step}
   \begin{split}
    &  \#\left(\mathcal{X}_{\mathcal{J}}, (\tilde{g}_{n+1}\circ \pi_{n+1})^*\norm{\cdot}_{N_{n+1}}, S_{N_{n+1}}\varphi, \varepsilon\right) \\
   &  < \left\{\left(\frac{1}{\varepsilon}\right)^{\left(D+\frac{3}{n}\right)N_n}\right\}^{|\mathcal{J}|/N_n} \cdot
    \left\{\left(\frac{1}{\varepsilon}\right)^{C_n + \norm{\varphi}_\infty}\right\}^{\left|[N_{n+1}]\setminus \mathcal{J}\right|}. 
   \end{split}
\end{equation}   
We have $|\mathcal{J}|\leq N_{n+1}$ and
\begin{equation*}
   \begin{split}
      \left|[N_{n+1}]\setminus \mathcal{J}\right| & \leq 2N_n \left|\partial_{\psi_n}(x)\cap [0,N_{n+1}]\right| + 2N_n  \\
     & < \frac{N_{n+1}}{n(C_n+\norm{\varphi}_\infty)}  \quad \text{by the condition (b) of the choice of $N_{n+1}$}.
    \end{split}
\end{equation*}      
Then the above (\ref{eq: covering number of X_J in the (n+1)-th step}) is bounded by 
\[  \left(\frac{1}{\varepsilon}\right)^{\left(D+\frac{3}{n}\right)N_{n+1} + \frac{N_{n+1}}{n}}
     =  \left(\frac{1}{\varepsilon}\right)^{\left(D+\frac{4}{n}\right)N_{n+1}}   . \]
The number of the choices of $\mathcal{J}\subset [N_{n+1}]$ is bounded by $2^{N_{n+1}}$.
Thus 
\begin{equation*}
      \#\left(\mathcal{X}, (\tilde{g}_{n+1}\circ \pi_{n+1})^*\norm{\cdot}_{N_{n+1}}, S_{N_{n+1}}\varphi, \varepsilon\right) 
       <  2^{N_{n+1}}  \left(\frac{1}{\varepsilon}\right)^{\left(D+\frac{4}{n}\right)N_{n+1}}.
\end{equation*}   
\end{proof}

From Lemma \ref{lemma: preparations on linear maps} (1), we can take $0<\varepsilon_{n+1}<\varepsilon_n/2$ such that 
for any $0<\varepsilon \leq \varepsilon_{n+1}$ and any linear map $g:P_{n+1}\to V^{N_{n+1}}$ with 
$g(P_{n+1})\subset B_1^\circ(V)^{N_{n+1}}$
\[  \#\left(g(\Delta),\norm{\cdot}_{N_{n+1}},\varepsilon\right) 
    <  \frac{1}{(\text{Number of simplexes of $P_{n+1}$})} \left(\frac{1}{\varepsilon}\right)^{\dim \Delta + \frac{1}{n+1}}  \]
for all simplexes $\Delta\subset P_{n+1}$.    

Let $g:P_{n+1}\to B_1^\circ(V)^{N_{n+1}}$ be a linear map and let $\Delta\subset P_{n+1}$ be a simplex.
Since $\pi_{n+1}$ is essential, we can find a point $x\in \pi_{n+1}^{-1}\left(O_{P_{n+1}}(\Delta)\right)$
with $\dim_{\pi_{n+1}(x)} P_{n+1}\geq \dim \Delta$.
From (\ref{eq: diameter of the preimage of open star of P_{n+1}}) and the condition (ii) of the choice of 
$\tau_{n+1}$
\[  \sup_{\pi_{n+1}^{-1}\left(O_{P_{n+1}}(\Delta)\right)} S_{N_{n+1}}\varphi \leq S_{N_{n+1}}\varphi(x) + \frac{N_{n+1}}{n+1}. \]
Then for $0<\varepsilon \leq \varepsilon_{n+1}$
\begin{equation*}
   \begin{split}
   & \left(\frac{1}{\varepsilon}\right)^{\sup_{\pi_{n+1}^{-1}\left(O_{P_{n+1}}(\Delta)\right)} S_{N_{n+1}}\varphi}
   \#\left(g(\Delta), \norm{\cdot}_{N_{n+1}}, \varepsilon\right)  \\
   &<   \frac{1}{(\text{Number of simplexes of $P_{n+1}$})} 
    \left(\frac{1}{\varepsilon}\right)^{S_{N_{n+1}}\varphi(x)+\dim \Delta + \frac{N_{n+1}+1}{n+1}} \\
  & \leq    \frac{1}{(\text{Number of simplexes of $P_{n+1}$})} 
    \left(\frac{1}{\varepsilon}\right)^{S_{N_{n+1}}\varphi(x)+\dim_{\pi_{n+1}(x)} P_{n+1} + \frac{N_{n+1}+1}{n+1}} \\
  & \leq  \frac{1}{(\text{Number of simplexes of $P_{n+1}$})} 
    \left(\frac{1}{\varepsilon}\right)^{\left(D+\frac{1}{n+1}\right)N_{n+1} + \frac{N_{n+1}+1}{n+1}} 
     \quad   \text{by (\ref{eq: local dimension bound on P_{n+1}})}  \\
  & \leq   \frac{1}{(\text{Number of simplexes of $P_{n+1}$})} 
    \left(\frac{1}{\varepsilon}\right)^{\left(D+\frac{3}{n+1}\right)N_{n+1}}.
     \end{split}
\end{equation*}   
Hence for any $0<\varepsilon\leq \varepsilon_{n+1}$ and any linear map $g:P_{n+1}\to B_1^\circ(V)^{N_{n+1}}$
\begin{equation} \label{eq: sum of the covering numbers of g(Delta)}
  \sum_{\Delta\subset P_{n+1}} \left(\frac{1}{\varepsilon}\right)^{\sup_{\pi_{n+1}^{-1}\left(O_{P_{n+1}}(\Delta)\right)} S_{N_{n+1}}\varphi}
   \#\left(g(\Delta), \norm{\cdot}_{N_{n+1}}, \varepsilon\right)  
    <   \left(\frac{1}{\varepsilon}\right)^{\left(D+\frac{3}{n+1}\right)N_{n+1}}.
\end{equation}

We define $g'_{n+1}: Q_{n+1}\to B^\circ_1(V)$ and $g_{n+1,\lambda}:P_{n+1}\to B_1^\circ(V)$ $(\lambda\in [N_{n+1}])$
as small perturbations of $\tilde{g}'_{n+1}$ and $\tilde{g}_{n+1,\lambda}$ respectively.
By Lemma \ref{lemma: preparations on linear maps} (2), we can assume that 
they satisfy Condition \ref{condition: dynamical PS theorem} (1).
From (\ref{eq: discrepancy of tilde{g}_{n+1}}) and (\ref{eq: f_n T^lambda and tilde{g}_{n+1,lambda} is close}) we can assume that 
the perturbations are so small that they satisfy 
\begin{equation} \label{eq: discrepancy between g_{n+1} and f_n}
  \begin{split}
   \norm{g'_{n+1}(\pi'_{n+1}(x))-f_n(x)} & < \min\left(\frac{\varepsilon_n}{8}, \frac{\delta_n}{2}\right),  \\
   \norm{g_{n+1,\lambda}(\pi_{n+1}(x))-f_n(T^\lambda x)} & < \min\left(\frac{\varepsilon_n}{8}, \frac{\delta_n}{2}\right).
   \end{split}
\end{equation}
Moreover, from Claim \ref{claim: covering number with respect to tilde{g}_{n+1}}, we can assume that 
$g_{n+1} := (g_{n+1,0}, \dots, g_{n+1,N_{n+1}-1})$ satisfies
\[ \#\left(\mathcal{X}, (g_{n+1}\circ \pi_{n+1})^*\norm{\cdot}_{N_{n+1}}, S_{N_{n+1}}\varphi, \varepsilon\right)
     < 2^{N_{n+1}} \left(\frac{1}{\varepsilon}\right)^{\left(D+\frac{4}{n}\right) N_{n+1}}  \]
for all $\varepsilon_{n+1}\leq \varepsilon \leq \varepsilon_n$.
On the other hand, from (\ref{eq: sum of the covering numbers of g(Delta)}), for $0<\varepsilon \leq \varepsilon_{n+1}$
\[  \sum_{\Delta\subset P_{n+1}}  \left(\frac{1}{\varepsilon}\right)^{\sup_{\pi_{n+1}^{-1}\left(O_{P_{n+1}}(\Delta)\right)} S_{N_{n+1}}\varphi}
   \#\left(g_{n+1}(\Delta), \norm{\cdot}_{N_{n+1}}, \varepsilon\right)  
    <   \left(\frac{1}{\varepsilon}\right)^{\left(D+\frac{3}{n+1}\right)N_{n+1}}. \]
Thus we have established Condition \ref{condition: dynamical PS theorem} (2)  and (3) for $(n+1)$-th step.
(Recall that the condition (2) is stronger than (3) over the region $0<\varepsilon \leq \varepsilon_{n+1}$.)
From Lemma \ref{lemma: dynamical tiling}, we can take a continuous function 
$\psi_{n+1}:\mathcal{X}\to [0,1]$ satisfying Condition \ref{condition: dynamical PS theorem} (4).
The map $f_{n+1}$ defined by (\ref{eq: continuous map f_n in dynamical PS theorem}) is a $(1/n)$-embedding with respect to $\mathbf{d}$
by Condition \ref{condition: dynamical PS theorem} (1).
Since ``$(1/n)$-embedding'' is an open condition, we can take $\delta_{n+1}>0$ 
satisfying Condition \ref{condition: dynamical PS theorem} (5).
From (\ref{eq: discrepancy between g_{n+1} and f_n}) 
\[  \norm{f_{n+1}(x)-f_n(x)} <  \min\left(\frac{\varepsilon_n}{8}, \frac{\delta_n}{2}\right). \]
This shows Condition \ref{condition: dynamical PS theorem} (6).
We have finished the constructions of all data for the $(n+1)$-th step.
\end{proof}

\vspace{0.5cm}

\address{ Masaki Tsukamoto \endgraf
Department of Mathematics, Kyoto University, Kyoto 606-8502, Japan}

\textit{E-mail address}: \texttt{masaki.tsukamoto@gmail.com}

\end{document}